\documentclass[a4paper, 11pt,reqno]{amsart}
\usepackage{amsthm,amssymb,amsmath,graphicx,psfrag,enumitem,mathrsfs}
\usepackage{mathtools}
\usepackage[foot]{amsaddr}
\usepackage[british]{babel}
\usepackage[babel]{microtype}
\usepackage[margin=1in]{geometry}

\usepackage[textsize=footnotesize]{todonotes}

\usepackage{algorithm}
\usepackage{algpseudocode}
\usepackage{tikz}

\usetikzlibrary{backgrounds}

\usepackage[numbers]{natbib}

\makeatletter
\def\NAT@spacechar{~}
\makeatother

\usepackage{hyperref}
\usepackage[capitalise]{cleveref}
\hypersetup{colorlinks=true,
   citecolor=blue,
   filecolor=blue,
   linkcolor=blue,
   urlcolor=blue
}

\crefname{figure}{Figure}{Figures}
\Crefname{figure}{Figure}{Figures}

\allowdisplaybreaks

\newtheorem{definition}{Definition}[section]

\newtheorem{proposition}[definition]{Proposition}
\newtheorem{theorem}[definition]{Theorem}
\newtheorem{corollary}[definition]{Corollary}
\newtheorem{lemma}[definition]{Lemma}

\newtheorem{conjecture}[definition]{Conjecture}

\numberwithin{equation}{section}

\newcommand{\comment}[1]{}

\renewcommand{\epsilon}{\varepsilon}

\newcommand{\eps}{\varepsilon}

\newcommand{\COMMENT}[1]{}

\title{Dirac's theorem for random regular graphs}

\address{School of Mathematics, University of Birmingham, 
Edgbaston, Birmingham, B15 2TT, United Kingdom.}

\author[P.~Condon]{Padraig Condon}
\email{pxc644@bham.ac.uk} 
\author[A.~Espuny D\'iaz]{Alberto Espuny D\'iaz}
\email{axe673@bham.ac.uk}
\author[A.~Gir\~{a}o]{Ant\'onio Gir\~{a}o}
\email{giraoa@bham.ac.uk}
\author[D.~K\"uhn]{Daniela K\"uhn}
\email{d.kuhn@bham.ac.uk}
\author[D.~Osthus]{Deryk Osthus}
\email{d.osthus@bham.ac.uk}

\thanks{This project has received partial funding from the European Research 
Council (ERC) under the European Union's Horizon 2020 research and innovation programme (grant agreement no. 786198, D.~K\"uhn and D.~Osthus).
The research leading to these results was also partially supported by the EPSRC, grant nos. EP/N019504/1 (A.~Gir\~ao and D.~K\"uhn) and EP/S00100X/1  (D.~Osthus), as well as the Royal Society and the Wolfson Foundation (D.~K\"uhn).}

\date{\today}

\begin{document}

\begin{abstract}
We prove a `resilience' version of Dirac's theorem in the setting of random regular graphs. 
More precisely, we show that, whenever $d$ is sufficiently large compared to $\eps>0$, a.a.s. the following holds: let $G'$ be any subgraph of the random $n$-vertex $d$-regular graph $G_{n,d}$ with minimum degree at least $(1/2+\eps)d$.
Then $G'$ is Hamiltonian.

This proves a conjecture of Ben-Shimon, Krivelevich and Sudakov.
Our result is best possible: firstly, the condition that $d$ is large cannot be omitted, and secondly, the minimum degree bound cannot be improved.
\end{abstract}
\maketitle
\thispagestyle{empty}

\section{Introduction}

The study of Hamiltonicity has been at the core of graph theory for the past few decades.
A graph $G$ is said to be \emph{Hamiltonian} if it contains a cycle which covers all of the vertices of $G$, and this is called a \emph{Hamilton cycle}.
It is well-known that the problem of determining whether a graph is Hamiltonian is NP-complete, and thus most results about Hamiltonicity deal with sufficient conditions which guarantee this property.
One of the most well-known examples is due to Dirac, who proved that any graph $G$ on $n\geq3$ vertices with minimum degree at least $n/2$ is Hamiltonian.

\subsection{Hamilton cycles in random graphs}

The search for Hamilton cycles in various models of random graphs has also been a driving force in the development of this theory.
The classical binomial model $G_{n,p}$, in which each possible edge is added to an $n$-vertex graph with probability $p$ independently of the other edges, has seen many results in this direction.
In particular, \citet{KS83} showed that $p=\log n/n$ is the `sharp' threshold for the existence of a Hamilton cycle.
This can be strengthened to obtain the following hitting time result.
Consider a random graph process as follows: given a set of $n$ vertices, add each of the $\binom{n}{2}$ possible edges, one by one, by choosing the next edge uniformly at random among those that have not been added yet.
In this setting, \citet{AKS85} and \citet{Bol84} independently proved that a.a.s.~the resulting graph becomes Hamiltonian as soon as its minimum degree is at least $2$.

The search for Hamilton cycles in other random graph models has proven more difficult.
In this paper we will deal with random \emph{regular} graphs: given $n,d\in\mathbb{N}$ such that $d<n$ and $nd$ is even, $G_{n,d}$ is chosen uniformly at random from the set of all $d$-regular graphs on $n$ vertices.
The study of this model is often more challenging than that of $G_{n,p}$ due to the fact that the presence and absence of edges in $G_{n,d}$ are correlated.
Several different techniques have been developed to deal with this model, such as the configuration model (see \cref{section:config}) or edge-switching techniques.
\citet{RW92} proved that $G_{n,3}$ is a.a.s.~Hamiltonian, and later extended this result to $G_{n,d}$ for any fixed $d\geq3$~\cite{RW94}.
This is in contrast to $G_{n,p}$, where the average degree must be logarithmic in $n$ to ensure Hamiltonicity.
These results were later generalised by \citet{CFR02} and \citet{KSVW01} for the case when $d$ is allowed to grow with $n$, up to $d\leq n-1$.
Many further results can be found in the recent survey of \citet{Frieze19}.

\subsection{Local resilience}

More recently, several extremal results have been translated to random graphs via the concept of local resilience.
The \emph{local resilience} of a graph $G$ with respect to some property $\mathcal{P}$ is the maximum number $r\in\mathbb{N}$ such that, for all $H\subseteq G$ with $\Delta(H)<r$, the graph $G\setminus H$ satisfies $\mathcal{P}$.
We say that $G$ is \emph{$r$-resilient} with respect to a property  $\mathcal{P}$ if the local resilience of $G$ is greater than $r$.
The systematic study of local resilience was initiated by \citet{SV08}, and the subject has seen a lot of research since.

Note that Dirac's theorem can be restated in this terminology to say that the local resilience of the complete graph $K_n$ with respect to Hamiltonicity is $\lfloor n/2\rfloor$.
This concept of local resilience then naturally suggests a generalisation of Dirac's theorem to random graphs.
In the binomial model, \citet{LS12} showed that, for any constant $\eps>0$, if $p\geq C\log n/n$ and $C$ is sufficiently large, then a.a.s.~$G_{n,p}$ is $(1/2-\eps)np$-resilient with respect to Hamiltonicity.
This improved on earlier bounds \cite{BKS11b,BKS11a,FK08,SV08}.
Very recently, \citet{Mont17} and independently \citet{NST17},  proved a hitting time result for the local resilience of $G_{n,p}$ with respect to Hamiltonicity.
In a different direction, \citet{CEKKO} considered `resilient' versions of P\'osa's theorem and Chv\'atal's theorem for $G_{n,p}$.

The resilience of random regular graphs with respect to Hamiltonicity is less understood.
\citet{BKS11b} proved that, for large (but constant) $d$, a.a.s.~$G_{n,d}$ is $(1-\eps)d/6$-resilient with respect to Hamiltonicity.
They conjectured that the true value should be closer to $d/2$.

\begin{conjecture}[\citet{BKS11b}]\label{conj: resilience}
For every $\eps>0$ there exists an integer $D=D(\eps)>0$ such that, for every fixed integer $d>D$, the local resilience of $G_{n,d}$ with respect to Hamiltonicity a.a.s.~lies in the interval $((1/2-\eps)d,(1/2+\eps)d)$. 
\end{conjecture}

They also suggested to study the same problem when $d$ is allowed to grow with $n$.
In this direction, \citet{SV08} showed that, for any fixed $\eps>0$, and for any $(n,d,\lambda)$-graph  $G$ (that is, a $d$-regular graph on $n$ vertices whose second largest eigenvalue in absolute value is at most $\lambda$) with $d/\lambda>\log^2n$, we have that $G$ is $(1/2-\eps)d$-resilient with respect to Hamiltonicity.
This, together with a result of \citet{KSVW01} and recent results of \citet{CGJ18}\COMMENT{Improving on previous results of \citet{BFSU99}.} and \citet{TY19} about the spectral gap of random regular graphs, implies that, for $\log^4n\ll d\leq n-1$, a.a.s.~$G_{n,d}$ is $(1/2-\eps)d$-resilient with respect to Hamiltonicity\COMMENT{The proof has to be broken into several chunks.
In order to cover the range $d=o(n^{2/3})$ we use the result of \citet{CGJ18}.
In order to cover the range $d=\Omega(n^{2/3})$ with $d\leq n/2$ we use the result of \citet{TY19}.
(Alternatively, we could use \cite{BFSU99} and \cite{TY19}, respectively, changing $n^{2/3}$ by $n^{1/2}$.)
Finally, the range for $n/2\leq d\leq (1-o(1/\log n))n$ can be covered by the result of \citet{KSVW01}.
(This result may seem to give a much larger range, but when applied in this setting actually leaves a gap when $d$ gets very close to linear, which forces us to use \cite{TY19}.)
Finally, to cover the remaining case when $d$ is allowed to be as large as $n-1$, the result follows by applying Dirac's theorem.}.
One can extend this to $d\gg\log n$ by combining a result of \citet{KV04} on joint distributions of binomial random graphs and random regular graphs with the result of \citet{LS12} about the resilience of $G_{n,p}$ with respect to Hamiltonicity. 

The study of local resilience has not been restricted to Hamiltonicity.
Other properties that have been considered include the containment of perfect matchings~\cite{CEKKO,NST17}, directed Hamilton cycles~\cite{FNNPS17,HSS16,Mont19}, cycles of all possible lengths~\cite{KLS10}, $k$-th powers of cycles~\cite{NST2}, bounded degree trees~\cite{BCS11}, triangle factors~\cite{BLS}, and bounded degree graphs~\cite{ABET,HLS12}.

\subsection{New results}

In this paper, we completely resolve \cref{conj: resilience}, as well as its extension to $d$ growing slowly with $n$ (recall that the case when $d \gg \log n$ is covered by earlier results).
This can be seen as a version of Dirac's theorem for random regular graphs.
Our main result gives the lower bound in \cref{conj: resilience}.

\begin{theorem}\label{thm: main}
For every $\eps>0$ there exists $D$ such that, for every $D<d\leq\log^2n$, the random graph $G_{n,d}$ is a.a.s.~$(1/2 - \eps)d$-resilient with respect to Hamiltonicity. 
\end{theorem}

While we do not try to optimise the dependency of $D$ on $\eps$, we remark that $D$ in \cref{thm: main} can be taken to be polynomial in $\eps^{-1}$.\COMMENT{$O(\eps^{-30})$.} \COMMENT{Antonio: I think that perhaps with a bit of effort, but not too much, we could show $O(\eps^{-2})$ is enough as a lower bound for $D$.} 
This is essentially best possible in the sense that \cref{thm: main} fails if  $d \le (2\eps)^{-1}$. 

\begin{theorem}\label{thm: counter}
For any odd $d > 2$, the random graph $G_{n,d}$ is not a.a.s.~$(d-1)/2$-resilient with respect to Hamiltonicity.  
\end{theorem}

Our proof also shows that $G_{n,d}$ is not a.a.s.~$(d- 1)/2$-resilient with respect to the containment of a perfect matching.
Moreover, one can adapt the proof of \cref{thm: counter} to show that, for every even $d$, the random graph $G_{n,d}$ is not a.a.s.~$d/2$-resilient with respect to Hamiltonicity (or the containment of a perfect matching).
It would also be interesting to obtain bounds on the resilience for small $d$.
In particular, here are some questions:
\begin{enumerate}[label=\roman*)]
    \item Given any fixed even $d$, determine whether the graph $G_{n,d}$ is a.a.s.~$(d/2-1)$-resilient with respect to Hamiltonicity.
    \item What is the likely resilience of $G_{n,4}$ with respect to Hamiltonicity or the containment of perfect matchings?
    Is a graph obtained from $G_{n,4}$ by removing any matching a.a.s.~Hamiltonian?
\end{enumerate}

Finally, we observe (as is well known) that the upper bound of $(1/2 + \eps)d$ in \cref{conj: resilience} follows easily from edge distribution properties of random regular graphs. 
Indeed, we note that for every $\epsilon>0$ there exists a constant $D$ such that for every $D \leq d\leq \log^2n$, a.a.s.\COMMENT{In fact, for every $\eps>0$, there is a constant $D$ such that for any $d\geq D$, every $d$-regular graph $G$ is not $(1/2+\eps)$-resilient with respect to being connected. To see this, consider a uniformly at random partition of $V(G)\coloneqq A \cup B$. Let us denote by $E_v$ the event the vertex $v$ has more than $(1/2+\eps)d$ neighbours in the opposite part. By Chernoff bounds, we know the probability of $E_v$ happening is at most $e^{O(-\eps^2d)}$. Since the event $E_v$ is independent of the system $\{E_{v'}: d(v,v')\geq 3\}$ we may apply Lovasz Local Lemma to deduce there exists a way to partition $V(G)$ where no vertex sends more than $(1/2+\eps)d$ edges to the opposite part.}~the graph $G=G_{n,d}$ has the property that between any two disjoint sets $A,B$ of size $\lfloor n/2 \rfloor$ and $\lceil n/2 \rceil$, respectively, the number of edges in $G[A,B]$ is a.a.s.~bounded from above by $(1/2+\epsilon/2)nd/2$ (see \cref{prop: edges}).\COMMENT{We apply \cref{prop: edges} with $R$ as the empty graph, $A$ as $A$ and for each $a \in A$ we let $Z_a$ be all edges from $a$ to $B$.
Note that $z = (n/2)(n/2) = n^2/4 > \eps n^2$.
It follows now that with probability at least $1- e^{-\eps^3 nd}$ we have that $e_G(A,B) < (1+\eps)zd/n  = (1/2+ \eps/2)nd/2$.
The result then follows by a union bound.}
Now let $A,B$ be a maximum cut in $G$.
Thus, $e_G(a,B)\geq d/2$ for all $a\in A$, and similarly for all $b\in B$.
If $|A| \neq |B|$, then by deleting the edges in $G[A]\cup G[B]$\COMMENT{Technically we use + for graphs, but only on the same vertex set. Maybe we should change this back to the union of the two sets of edges here?}, the remaining graph is not Hamiltonian since it forms an unbalanced bipartite graph.
If $|A|=|B|$, then by the above property, there must exist a vertex $x \in A$ such that $e_G(x,B)\leq (1/2+\eps/2)d$. 
Let $A'\coloneqq A\setminus \{x\}$ and $B'\coloneqq B\cup \{x\}$. 
As before, by deleting the edges in $G[A']\cup G[B']$, we obtain a graph which is not Hamiltonian. 

\subsection{Organisation of the paper}

The remainder of the paper is organised as follows.
In \cref{section: sketch} we give a sketch of the proof of \cref{thm: main}.
In \cref{section: prelims} we collect notation, some probabilistic tools, and observations about the configuration model.
\Cref{section: expansion} is devoted to proving different edge-distribution and expansion properties of random regular graphs and their subgraphs, and the proof of \cref{thm: main} is given in \cref{section: proof}, using all the techniques that have been introduced before.
Finally, we prove \cref{thm: counter} in \cref{section: counter}.


\section{Outline of the proof of \texorpdfstring{\cref{thm: main}}{Theorem 1.2}}\label{section: sketch}

Consider $G = G_{n,d}$.
Let $H \subseteq G$ be such that $\Delta(H) \le (1/2-\eps)d$ and let $G'\coloneqq G\setminus H$.
We will prove that $G'$ contains a `sparse' spanning subgraph $R$ which has strong edge expansion properties.
These properties will then be used to provide a lower bound on the number of edges in $G$ whose addition would make $R$ Hamiltonian, or increase the length of a longest path in $R$ (such edges are commonly called `boosters', see e.g.~\cite{Kri16}).
We then argue that some of these edges must in fact be retained when passing to $G'$.
We then add such edges to $R$ and iterate the above process (at most $n$ times) until $R$ becomes Hamiltonian.

More specifically, as a preliminary step we `thin' the graph $G'$, that is, we take a subgraph $R \subseteq G'$ with $\Delta(R) \le \delta d$, for some $\delta \ll \eps$.
As described above, we consider a longest path in $R$ and then argue that it can be extended via edges in $G'\setminus R$.
The fact that $R$ is relatively `sparse' with respect to $G'$ will be important when calculating union bounds over all graphs $R$ of this type, at a later stage in the proof (this idea was introduced by \citet{BKS11b}).

Given many paths of maximum length and with different endpoints in $R$, it follows that there will be many edges whose addition will increase the length of a longest path (or make $R$ Hamiltonian).
A theorem of P\'osa implies that graphs with strong expansion properties will indeed contain many of such paths.
These expansion properties are captured by the notion of a 3-expander (see \cref{def: expand}).
Therefore, we wish to show that our thinned graph $R$ can be chosen to be a 3-expander.
This is one point where working with the random graph $G_{n,d}$ proves more difficult than working with $G_{n,p}$, due to the fact that the appearance of edges in $G_{n,d}$ is correlated.

The next step is to provide a lower bound on the number of edges whose addition to $R$ would increase the length of a longest path (or make $R$ Hamiltonian).
Here we further develop an approach of \citet{Mont17} who, instead of considering single edges that would bring $R$ closer to being Hamiltonian, considered `booster' edge pairs whose addition would yield the same result.
For example, if $R$ is connected and $P$ is a longest path in $R$ with endpoints $x$ and $y$, and $ab$ is an edge of $P$ (with $b$ closer to $y$ on $P$), then $\{ya, xb\}$ is a booster pair.
The main advantage of considering such pairs of edges is that it results in a much larger set of boosters for $R$.
More precisely, we show the existence of another thinned graph $F \subseteq G' \setminus R$ for which each booster we consider is of the form $\{e, e'\}$, where $e \in E(F)$ and $e'\in E(G')$ (see \cref{cor: boost}).

Finally, we can complete the proof of the main theorem by iteratively adding booster pairs to the thinned graph $R$, increasing the length of a maximum path in each step until $R$ becomes Hamiltonian. 
Two points are important here as to why we can iterate this process.
First, proving the existence of boosters (see \cref{lem: rand}) involves a union bound over all pairs of thinned graphs $R$ and $F$.
To bound this efficiently, we need that both $R$ and $F$ are relatively `sparse' with respect to $G'$.
But in each step we only add two booster edges to $R$, so it remains sparse.
Secondly, we take special care to ensure that no vertex is contained in too many of the boosters we add to $R$, ensuring that its degree in successive iterations remains small.
This process terminates after at most $n$ iterations, resulting in a graph $R'\subseteq G'$ which is Hamiltonian.


\section{Preliminaries}\label{section: prelims}

\subsection{Notation}

For $n \in \mathbb{N}$, we denote $[n] \coloneqq \{1, \ldots , n\}$.
Given any set $S$, we denote $S^{(2)}\coloneqq\{\{s_1,s_2\}:s_1,s_2\in S, s_1\neq s_2\}$.
The parameters which appear in hierarchies are chosen from right to left.
That is, whenever we claim that a result holds for $0 < a \ll b \le 1$, we mean that there exists a non-decreasing function $f\colon [0, 1) \to [0,1)$ such that the result holds for all $a>0$ and all $b \le 1$ with $a \le f(b)$.
We will not compute these functions explicitly.

Throughout this paper, the word \emph{graph} will refer to a simple, undirected graph.
Whenever the graphs are allowed to have parallel edges or loops, we will refer to these as \emph{multigraphs}.
Given any (multi)graph $G=(V,E)$ and sets $A,B\subseteq V$, we will denote the (multi)set of edges of $G$ spanned by $A$ as $E_G(A)$, and the (multi)set of edges of $G$ having one endpoint in $A$ and one endpoint in $B$ as $E_G(A,B)$.
The number of such edges will be denoted by $e_G(A)$ and $e_G(A,B)$, respectively.
We will also write $e(G)$ for $e_G(V)$.
Given two (multi)graphs $G_1$ and $G_2$ on the same vertex set $V$, we write $G_1+G_2\coloneqq(V,E(G_1)\cup E(G_2))$, where the union represents set union for graphs and multiset union for multigraphs.
When $G_1$ and $G_2$ are graphs, we write $G_1\setminus G_2\coloneqq(V,E(G_1)\setminus E(G_2))$.
Given any vertex $v\in V$, we will denote the set of vertices which are adjacent to $v$ in $G$ by $N_G(v)$.
We define $N_G(A)\coloneqq\bigcup_{v\in A}N_G(v)$.
The \emph{degree} of vertex $v$ in a multigraph $G$ is $d_G(v)\coloneqq|\{e\in E(G):v\in e\}|+|\{e\in E(G):e=vv\}|$ (i.e.~each loop at $v$ contributes two to $d_G(v)$).
We denote $\Delta(G)\coloneqq\max_{v\in V}d_G(v)$ and $\delta(G)\coloneqq\min_{v\in V}d_G(v)$.
The (multi)graph $G$ is said to be \emph{$d$-regular} for some $d\in\mathbb{N}$ if all vertices have degree $d$.
Given a multigraph $G$ on $[n]$, we refer to the vector $\mathbf{d}=(d_G(1),\ldots,d_G(n))$ as its \emph{degree sequence}. 
In general, a vector $\mathbf{d}=(d_1,\ldots,d_n)$ with $d_i\in\mathbb{Z}_{\geq0}$ for all $i\in[n]$ is called \emph{graphic} if there exists a graph on $n$ vertices with degree sequence $\mathbf{d}$ (note that, as long as $\sum_{i=1}^nd_i$ is even, there is always a multigraph with degree sequence $\mathbf{d}$).
Given a graph $G$ and a real number $\alpha>0$, let $\mathcal{H}_\alpha(G)$ be the collection of all spanning subgraphs $H \subseteq G$ for which $d_H(v) \leq \alpha  d_G(v)$, for all $v \in V(G)$.

We will use $\mathcal{G}_{n,d}$ to denote the set of all $d$-regular graphs on vertex set $[n]$, and $G_{n,d}$ will denote a graph chosen from $\mathcal{G}_{n,d}$ uniformly at random.
Whenever we use this notation, we implicitly assume that $nd$ is even.
In more generality, given a graphic degree sequence $\mathbf{d}=(d_1,\ldots,d_n)$, we will denote the collection of all graphs on vertex set $[n]$ with degree sequence $\mathbf{d}$ by $\mathcal{G}_{n,\mathbf{d}}$, and $G_{n,\mathbf{d}}$ will denote a graph chosen from $\mathcal{G}_{n,\mathbf{d}}$ uniformly at random.

We use \emph{a.a.s.}~as an abbreviation for \emph{asymptotically almost surely}. 
Given a sequence of events $\{\mathcal{E}_n\}_{n\in\mathbb{N}}$, whenever we claim that $\mathcal{E}_n$ holds a.a.s., we mean that the probability that $\mathcal{E}_n$ holds tends to $1$ as $n$ tends to infinity. 
For the purpose of clarity, we will ignore rounding issues when dealing with asymptotic statements.
By abusing notation, given $p\geq0$ and $n\in\mathbb{N}$, we write $\mathrm{Bin}(n,p)$ for the binomial distribution with parameters $n$ and $\min\{p,1\}$.


\subsection{Probabilistic tools}\label{section:chernoff}

We will need the following Chernoff bound (see e.g.~\cite[Corollary 2.3]{JLR}).

\begin{lemma}\label{lem: Chernoff}
Let $X$ be the sum of $n$ independent Bernoulli random variables and let $\mu \coloneqq \mathbb{E}[X]$.
Then, for all $0\le \delta \le 1$ we have that $\mathbb{P}[|X-\mu|\geq\delta\mu]\leq2e^{-\delta^2\mu/3}$.
\end{lemma}

The following bound will also be used repeatedly (see e.g.~\cite[Theorem A.1.12]{AS16})\COMMENT{The theorem in \emph{The Probabilistic Method} reads as follows: 
If $X\sim\mathrm{Bin}(n,p)$ and $\beta>1$, then $\mathbb{P}[X\geq(\beta-1)np]\leq(e^{\beta-1}\beta^{-\beta})^{np}$.
It is easy to see that \cref{lem: betaChernoff} follows from this:
\[\mathbb{P}[X\geq\beta np]\leq\mathbb{P}[X\geq(\beta-1)np]\leq(e^{\beta-1}\beta^{-\beta})^{np}=\left(\frac{e}{\beta}\right)^{\beta np}\cdot\frac{1}{e^{np}}\leq\left(\frac{e}{\beta}\right)^{\beta np}.\]}.

\begin{lemma}\label{lem: betaChernoff}
Let $X\sim\mathrm{Bin}(n,p)$, and let $\beta>1$.
Then, $\mathbb{P}[X\geq\beta np]\leq\left(e/\beta\right)^{\beta np}$.
\end{lemma}

Given any sequence of random variables $X=(X_1,\ldots,X_n)$ taking values in a set $A$ and a function $f\colon A^n\to\mathbb{R}$, for each $i\in[n]\cup\{0\}$ define $Y_i\coloneqq\mathbb{E}[f(X)\mid X_1,\ldots,X_i]$.
The sequence $Y_0,\ldots,Y_n$ is called the \emph{Doob martingale} for $f$.
All the martingales that appear in this paper will be of this form.
To deal with them, we will need the following version of the well-known Azuma-Hoeffding inequality.

\begin{lemma}[Azuma's inequality \cite{Azu67,Hoef63}]\label{lem: Azuma}
Let $X_0,X_1,\ldots$ be a martingale and suppose that $|X_i-X_{i-1}|\leq c_i$ for all $i\in\mathbb{N}$.
Then, for any $n,t\in\mathbb{N}$,
\[\mathbb{P}[|X_n-X_0|\geq t]\leq2\exp\left(\frac{-t^2}{2\sum_{i=1}^nc_i^2}\right).\]
\end{lemma}

Finally, the L\'ovasz local lemma will be useful. 
Let $\mathfrak{E} \coloneqq \{\mathcal{E}_1, \mathcal{E}_2, \ldots, \mathcal{E}_m\}$ be a collection of events. 
A \emph{dependency graph} for $\mathfrak{E}$ is a graph $H$ on vertex set $[m]$ such that, for all $i\in[m]$, $\mathcal{E}_i$ is mutually independent of $\{\mathcal{E}_j:j\neq i, j\notin N_H(i)\}$, that is, if $\mathbb{P}[\mathcal{E}_i]=\mathbb{P}[\mathcal{E}_i\mid\bigwedge_{j\in J}\mathcal{E}_j]$ for all $J\subseteq[m]\setminus(N_H(i)\cup\{i\})$.
We will use the following version of the local lemma (it follows e.g.~from~\cite[Lemma 5.1.1]{AS16}).

\begin{lemma}[L\'ovasz local lemma]\label{lem: LLL}
Let $\mathfrak{E} \coloneqq \{\mathcal{E}_1, \mathcal{E}_2, \ldots, \mathcal{E}_m\}$ be a collection of events and let $H$ be a dependency graph for $\mathfrak{E}$.
Suppose that $\Delta(H)\leq d$ and $\mathbb{P}[\mathcal{E}_i]\leq p$ for all $i\in[m]$.
If $ep(d+1)\leq1$, then 
\[\mathbb{P}\left[\bigwedge_{i=1}^m\overline{\mathcal{E}_i}\right]\geq(1-ep)^m.\]
\end{lemma}


\subsection{The configuration model}\label{section:config}

We will work with the \emph{configuration model} introduced by \citet{Bol80}, which can be used to sample $d$-regular graphs uniformly at random.
In more generality, it can be used to produce graphs with any given graphic degree sequence $\mathbf{d}$.
The process to generate such graphs is as follows.

Given $n\in\mathbb{N}$ and a degree sequence $\mathbf{d}=(d_1,\ldots,d_n)$ with $m\coloneqq\sum_{i=1}^nd_i$ even, consider a set of $m$ vertices labelled as $x_{ij}$ for $i\in[n]$ and $j\in[d_i]$.
For each $i\in[n]$, we call the set $\{x_{ij}:j\in[d_i]\}$ the \emph{expanded set} of $i$.
Similarly, for any $X\subseteq[n]$, we call the set $\{x_{ij}:i\in X,j\in[d_i]\}$ the \emph{expanded set} of $X$.
Choose uniformly at random a perfect matching $M$ covering the expanded set of $[n]$.
Then, obtain a multigraph $\varphi(M)=([n],E)$ by letting $E$ be the following multiset: for each edge $e\in M$, consider its endpoints $e=x_{ij}x_{k\ell}$, for some $i,k\in[n]$, $j\in[d_i]$ and $\ell\in[d_k]$, and add $ik$ to $E$ (if $i=k$, this adds a loop to $E$).

When we consider a multigraph $G$ obtained via this configuration model, this will be denoted by $G\sim\mathcal{C}_{n,\mathbf{d}}$.
In particular, when we obtain a $d$-regular multigraph via the configuration model, we will denote this by $G\sim\mathcal{C}_{n,d}$.
We refer to the possible perfect matchings on the expanded set of $[n]$ as \emph{configurations}, and we will denote a configuration obtained uniformly at random by $M\sim\mathcal{C}_{n,\mathbf{d}}^{*}$.
By abusing notation, we will sometimes also use $\mathcal{C}_{n,\mathbf{d}}^{*}$ to denote the set of all configurations with parameters $n$ and $\mathbf{d}$.
In order to easily distinguish the setting of graphs from that of configurations, we will call the elements of the expanded sets \emph{points}, and each element in a configuration will be called a \emph{pairing}.

The above process may produce a multigraph with loops and/or multiple edges.
However, if $\mathbf{d}$ is a graphic degree sequence, then, when conditioning on the resulting multigraph being simple, the configuration model yields a graph $G\in\mathcal{G}_{n,\mathbf{d}}$ chosen uniformly at random.
The following proposition bounds the probability that this happens, and can be proved similarly to (part of) a result of \citet[Lemma~7]{CFR02}. 
For completeness, we include a full proof in \cref{appendix}.
It will be useful when analysing the distribution of edges in $G_{n,d}$ via the configuration model.
\COMMENT{So I think a simple switching argument gives it. Indeed, let $A_t$ be the set of configurations $C$ with exactly $t\geq 10d^3$ clash edges i.e the multigraph $\phi(C)+R$ is not simple with $t$ multi-edges. Now, let us consider the bipartite graph between $A=A_t$ to $\cup_{i\leq t-1} A_i=B$. For each such configuration $C$ there are $t\cdot (1-\delta)dn-4(d+\delta d)$ edges to $B$. Each vertex in $B$ has degree at most $dn\cdot 2(d^2+d)$. This implies that $|A_t|\leq |B|/d$, and so $\mathcal{P}[\cup_{t\geq 10d^3} A_t]\leq 1/d$. Now, for smaller values of $t$ we get that $|B|\geq |A_t|/d $ (substitute $t=1$ in the above expression), then we get that $|A_0|\geq (1-1/d)\cdot (1/d)^{10d^3}$, and this is enough.} 

\begin{proposition}\label{prop: small subgraph}
Let $0<\delta < 1/10$.
Let $d \le \log^2 n$ be a positive integer and let $R$ be a graph on vertex set $[n]$ with degree sequence $\mathbf{d'}=(d_1,\ldots,d_n)$ such that $d_i < \delta d$ for all $i \in [n]$. 
Let $\mathbf{d}\coloneqq (d-d_1, \dots, d -d_{n})$ and let $F\sim\mathcal{C}_{n,\mathbf{d}}$.
Then, if $n$ is sufficiently large,
\[\mathbb{P}[R + F\text{ is simple}]\geq e^{-3d^2}.\]
\end{proposition}

Note that, by choosing $R$ to be the empty graph on $n$ vertices, we obtain a lower bound on the probability that the multigraph obtained by a random configuration is simple.

When studying the configuration model, it will be useful to consider the following process to generate $M\sim\mathcal{C}_{n,\mathbf{d}}^*$.
Let $\mathbf{d}=(d_1,\ldots,d_n)$ and suppose that $m\coloneqq\sum_{i=1}^nd_i$ is even.
Label the points of the expanded set of $[n]$ in any arbitrary order, $x_1,\ldots,x_m$, and identify them naturally with the set $[m]$.
Start with an empty set of pairings $M_0$.
Inductively, for each $i\in[m]$, if $i$ is covered by $M_{i-1}$, let $M_i\coloneqq M_{i-1}$; otherwise, choose a point $j\in[m]\setminus(V(M_{i-1})\cup\{i\})$ uniformly at random and define $M_i\coloneqq M_{i-1}\cup\{ij\}$.
We sometimes refer to $M_i$ as the $i$-th \emph{partial configuration}.
Finally, let $M\coloneqq M_m$.
It is clear that the resulting configuration $M$ is generated uniformly at random, independently of the labelling of the expanded set of $[n]$.

We will often be interested in bounding the number of edges in $G_{n,d}$ between two sets of vertices.
For this, it will be useful to consider binomial random variables that stochastically dominate the number of edges.
We formalise this via the following lemma.

\begin{lemma}\label{lem: StochDom}
Let $n,d\in\mathbb{N}$ with $d<n$, and let $\delta\in[0,1)$.
Let $\mathbf{d}=(d_1,\ldots,d_n)$ with $\sum_{i=1}^nd_i$ even be such that $(1-\delta)d\leq d_i\leq d$ for all $i\in[n]$.
Let $G\sim\mathcal{C}_{n,\mathbf{d}}$ and let $A,B\subseteq[n]$ be any (not necessarily disjoint) sets of vertices such that $2|A|<(1-\delta)n$.
Then, the random variable $e_G(A,B)$ is stochastically dominated by a random variable $X\sim\mathrm{Bin}(\sum_{a\in A}d_a,|B|/((1-\delta)n-2|A|))$.
\end{lemma}

\begin{proof}
Let $t\coloneqq\sum_{a\in A}d_a$.
Let $\mathcal{X}$, $A'$ and $B'$ be the expanded sets of $[n]$, $A$ and $B$, respectively.
Label the points of $\mathcal{X}$ so that all the points in $A'$ come first, that is, $A'=\{x_1,\ldots,x_t\}$.
Generate a random configuration $M\sim\mathcal{C}_{n,\mathbf{d}}^*$ following this labelling.
Then, $e_G(A,B)$ is the number of pairings in $M$ with one endpoint in $A'$ and the other in $B'$, and we will estimate the probability that each pairing added to $M$ contributes to $e_G(A,B)$.

First, note that all pairings added after $M_t$ do not contribute to $e_G(A,B)$, as they do not have an endpoint in $A'$.
For each $i\in[t]$, define an indicator random variable $X_i$ which takes value $1$ if $M_i\neq M_{i-1}$ and $e=x_iy\in M_i\setminus M_{i-1}$ is such that $y\in B'$, and $0$ otherwise, so that $e_G(A,B)=\sum_{i\in[t]}X_i$.
Observe that, in the above process, the bound
\[\mathbb{P}[X_i=1\mid M_i\neq M_{i-1}]\leq\frac{|B|}{(1-\delta)n-2|A|}\]
holds for all $i\in[t]$ since at every step of the process there are at most $|B|d$ points available in $B'$ and at least $(1-\delta)nd-2|A|d$ points available in $\mathcal{X}\setminus(V(M_{i-1})\cup\{x_i\})$.
On the other hand, $\mathbb{P}[X_i=1\mid M_i=M_{i-1}]=0$, so given $M_0,M_1,\ldots,M_{i-1}$, each $X_i$ is stochastically dominated by a Bernoulli random variable $Y_i$ with parameter $|B|/((1-\delta)n-2|A|)$.
By summing over all $i\in[t]$, we conclude that $e_G(A,B)$ is stochastically dominated by $X\sim\mathrm{Bin}(t,|B|/((1-\delta)n-2|A|))$.
\end{proof}


\section{On the existence of a sparse 3-expander}\label{section: expansion}

\begin{definition}\label{def: expand}
An $n$-vertex graph $G$ is called a \emph{$3$-expander} if it is connected and, for every $S \subseteq [n]$ with $|S| \le n/400$, we have $|N_G(S)| \ge 3|S|$.
\end{definition}

In order to give bounds on the distribution of edges in $G_{n,d}$ we will use an edge-switching technique, first introduced by \citet{MW90B}.
We will consider the following switching.

\begin{definition}\label{def: switch}
Let $G=(V, E)$ and $G' = (V, E')$ be two multigraphs on the same vertex set such that $|E| = |E'|$.
We write $G\sim G'$ if there exist $u_1u_2, v_1v_2 \in E$ such that $E' = (E \setminus \{u_1u_2, v_1v_2\}) \cup \{u_1v_1, u_2v_2\}$.
\end{definition}

The following lemma bounds the probability that certain variables on configurations deviate from their expectation.

\begin{lemma}\label{lem: mart}
Let $\mathbf{d}=(d_1,\ldots,d_n)$ be a degree sequence with $d_i \le \log^2 n$ for all $i\in[n]$, and  such that $\sum_{i=1}^nd_i$ is even.
Let $\Delta \coloneqq \max_{i \in [n]} \{d_i\}$.
Let $c>0$ and let $X$ be a random variable on $\mathcal{C}^*_{n,\mathbf{d}}$ such that, for every pair of configurations $M\sim M'$, we have $|X(M)-X(M')| \le c$.
Then, for all $\eps >0$, 
\[\mathbb{P}[|X - \mathbb{E}[X]| \ge \eps\mathbb{E}[X]] \le 2e^{- \frac{\eps^2\mathbb{E}[X]^2}{2\Delta nc^2}}.\]
\end{lemma}

\begin{proof}
Let $m\coloneqq \sum_{i=1}^{n} d_i$.
Fix any labelling $x_1,\ldots,x_m$ of the expanded set of $[n]$.
Let $M \sim \mathcal{C}_{n,\mathbf{d}}^*$ be generated following this labelling.
Let the partial configurations of $M$ be $M_0,\ldots,M_m$.
For each $i\in[m]\cup\{0\}$, let 
\[Y_i(M) \coloneqq \mathbb{E}[X(M) \mid M_i]=\mathbb{E}[X(M) \mid M_0,\ldots,M_i].\]
It follows that the sequence $Y_0(M), Y_1(M), \dots, Y_{m}(M)$ is a Doob martingale, where $Y_0(M) = \mathbb{E}[X]$ and $Y_{m}(M) = X(M)$.
We will now show that the differences of this martingale are bounded by $c$.

For any $i\in[m]$, if $M_i=M_{i-1}$, then $Y_i(M)=Y_{i-1}(M)$ and there is nothing to prove, so assume that $M_i\neq M_{i-1}$, that is, when generating the $i$-th partial configuration, the $i$-th point does not lie in any of the previous pairings.
For each $j\in[m]\setminus(V(M_{i-1})\cup\{i\})$, let $\mathcal{M}_j$ be the set of configurations which contain $M_{i-1}$ as well as $ij$.
It is easy to see that for each $k\in[m]\setminus(V(M_{i-1})\cup\{i\})$ there is a bijection $g_{j,k}$ between $\mathcal{M}_j$ and $\mathcal{M}_k$ so that $g_{j,k}(M')\sim M'$ for all $M'\in\mathcal{M}_j$\COMMENT{For all points $j, k\in[m]\setminus(V(M_{i-1})\cup\{i\})$ with $j\neq k$, for each $M^1 \in \mathcal{M}_j$, there exists a unique point $\ell\in[m]\setminus(V(M_{i-1})\cup\{i,j,k\})$ such that $\{k,\ell\}\in M^1$.
By performing the switching $\{\{ i,j\},\{k,\ell\}\} \rightarrow \{\{ i,k\},\{j,\ell\}\}$, we obtain a unique configuration $M^2 \in \mathcal{M}_k$.
This gives a bijection between $\mathcal{M}_j$ and $\mathcal{M}_k$, as required.}.
Fix $j\in[m]\setminus(V(M_{i-1})\cup\{i\})$, let $N\coloneqq|\mathcal{M}_j|$ and label the configurations in $\mathcal{M}_j$ as $M_{j,1},\ldots,M_{j,N}$.
For all $k\in[m]\setminus(V(M_{i-1})\cup\{i,j\})$, label $\mathcal{M}_k$ by $M_{k,\ell}\coloneqq g_{j,k}(M_{j,\ell})$ for each $\ell\in[N]$.
By assumption, we have $|X(M_{j,\ell})-X(M_{k,\ell})| \le c$ for all distinct $j,k\in[m]\setminus(V(M_{i-1})\cup\{i\})$ and $\ell\in[N]$.
Using this, it is easy to conclude that $|Y_i(M)-Y_{i-1}(M)|\leq c$.\COMMENT{We have that 
\[Y_{i}(M)=\frac1N\sum_{t=1}^NX(M_{\ell,t})=\frac{1}{|[m]\setminus(V(M_{i-1})\cup\{i\})|}\frac1N\sum_{t=1}^N|[m]\setminus(V(M_{i-1})\cup\{i\})|X(M_{\ell,t})\]
for whichever $\ell$ is paired with $i$ in $M_i$.
On the other hand, 
\[Y_{i-1}(M)=\frac{1}{|[m]\setminus(V(M_{i-1})\cup\{i\})|}\sum_{j\in[m]\setminus(V(M_{i-1})\cup\{i\})}\frac1N\sum_{t=1}^NX(M_{j,t}).\]
It follows that
\begin{align*}
    |Y_{i}(M)-Y_{i-1}(M)|&=\left|\frac{1}{|[m]\setminus(V(M_{i-1})\cup\{i\})|}\sum_{j\in[m]\setminus(V(M_{i-1})\cup\{i\})}\frac1N\sum_{t=1}^N(X(M_{j,t})-X(M_{\ell,t}))\right|\\
    &\leq\frac{1}{|[m]\setminus(V(M_{i-1})\cup\{i\})|}\sum_{j\in[m]\setminus(V(M_{i-1})\cup\{i\})}\frac1N\sum_{t=1}^N|X(M_{j,t})-X(M_{\ell,t})|\leq c.
\end{align*}
}

The statement now follows by \cref{lem: Azuma}.
\end{proof}

The following proposition implies that the distribution of edges in $G_{n,d}$ behaves roughly as in a binomial random graph $G_{n,d/n}$, even after conditioning on the containment of some `sparse' subgraph.

\begin{proposition}\label{prop: edges}
For every $0<\eps\leq1/2$ there exists $\delta>0$ such that the following holds.
Let $d \le \log^2 n$ be a positive integer and let $G = G_{n,d}$.
Let $R$ be a graph on vertex set $[n]$ with $\Delta(R) < \delta d$.
Moreover, let $A \subseteq [n]$ and, for each $a \in A$, let $Z_a \subseteq [n]^{(2)} \setminus E(R)$ be a collection of edges incident to $a$ such that $z \coloneqq \sum_{a\in A}|Z_a|$ satisfies $z> \eps n^2$. 
Then, 
\[\mathbb{P}\left[\bigg|\sum_{a\in A}|Z_a \cap E(G)| - \frac{zd}{n}\bigg| \ge \eps\frac{zd}{n} \,\middle\vert\, R \subseteq G\right] \le e^{-(\eps/10)^4 nd}.\]
\end{proposition}

\begin{proof}
Let $0< \delta \ll \eps$.
For each $i\in[n]$, let  $d_i \coloneqq d - d_R(i) > (1-\delta)d$, and let $\mathbf{d}\coloneqq(d_1,\ldots,d_n)$.
Let $M \sim\mathcal{C}_{n,\mathbf{d}}^*$ and let $F=\varphi(M)$, so that $F \sim\mathcal{C}_{n,\mathbf{d}}$ and $R+F$ is a $d$-regular multigraph.
By \cref{lem: StochDom}, for each $a\in A$, the random variable $Y_a\sim\mathrm{Bin}(d_a, (n-|Z_a|)/((1-\delta)n - 2))$ stochastically dominates $e_F(a, [n]\setminus (V(Z_a)\setminus\{a\}))$\COMMENT{Apply \cref{lem: StochDom} with $a$ playing the role of $A$ and $[n]\setminus(V(Z_a)\setminus\{a\})$ playing the role of $B$.}.
Let $Z(F) \coloneqq \sum_{a\in A}|Z_a \cap E(F)|$.

Note that $\mathbb{E}[Y_a] < (1+\eps^3)d_a(n-|Z_a|)/n$\COMMENT{$\mathbb{E}[Y_a]=d_a(n-|Z_a|)/((1-\delta)n - 2)<d_a(n-|Z_a|)/((1-2\delta)n)<(1+\eps^3)d_a(n-|Z_a|)/n$.} for all $a\in A$.
It then follows that $\mathbb{E}[|Z_a \cap E(F)|] \geq d_a - \mathbb{E}[Y_a] \ge (1+\eps^3)|Z_a|d_a/n - \eps^3 d_a$\COMMENT{We have that 
\begin{align*}
    \mathbb{E}[|Z_a \cap E(F)|]&=\mathbb{E}[d_a-e_F(a, [n]\setminus (V(Z_a)\setminus\{a\}))]\geq d_a-\mathbb{E}[Y_a]>d_a - (1+\eps^3)d_a(n-|Z_a|)/n\\
    &=\frac{d_a|Z_a|}{n}-\eps^3d_a+\eps^3\frac{d_a|Z_a|}{n} = (1+\eps^3)\frac{d_a|Z_a|}{n} - \eps^3 d_a.
\end{align*}}.
Therefore, we have $\mathbb{E}[Z(F)] \ge (1-\eps^2)zd/n$\COMMENT{$\mathbb{E}[Z(F)]=\sum_{a\in A}\mathbb{E}[|Z_a \cap E(F)|]\geq(1+\eps^3)(1-\delta)d/n\sum_{a\in A}|Z_a| - \eps^3nd \geq((1+\eps^3)(1-\delta)-\eps^2)zd/n\geq (1-\eps^2)zd/n$, where in the second inequality we use the bound on $z$, which results in $\eps^3nd\leq\eps^2zd/n$.}.
Now, note that $|Z(F)-Z(F')|\leq8$\COMMENT{Each switching affects a maximum of four edges, and each edge may appear in a maximum of two instances of $Z_a$.} when $F\sim F'$. 
Let $Z'\colon \mathcal{C}^*_{n,\mathbf{d}} \rightarrow \mathbb{Z}$ be such that $Z'(M) = Z(F)$ whenever $\varphi(M) = F$.
It follows that $|Z'(M)-Z'(M')|\leq8$\COMMENT{This follows since $M\sim M'$ implies that $\varphi(M) \sim \varphi(M')$.} when $M\sim M'$.
Moreover, $\mathbb{E}[Z'(M)]=\mathbb{E}[Z(F)]$.
Therefore, we can apply \cref{lem: mart} to obtain \COMMENT{We have that 
\begin{align*}
    \mathbb{P}\left[Z'(M)\leq(1-\eps)\frac{zd}{n}\right]&=\mathbb{P}\left[Z'(M)\leq\frac{1-\eps}{1-\eps^2}(1-\eps^2)\frac{zd}{n}\right]\leq\mathbb{P}\left[Z'(M)\leq\frac{1-\eps}{1-\eps^2}\mathbb{E}[Z'(M)]\right]\\
    &\leq\mathbb{P}\left[|Z'(M)-\mathbb{E}[Z'(M)]|\geq\left(1-\frac{1-\eps}{1-\eps^2}\right)\mathbb{E}[Z'(M)]\right].
\end{align*}
Applying \cref{lem: mart}, in the exponent we get 
\[-\left(1-\frac{1-\eps}{1-\eps^2}\right)^2\frac{\mathbb{E}[Z'(M)]^2}{2\cdot8^2\max\{d_i\}n}\leq-\left(\frac{\eps(1-\eps)}{1-\eps^2}\right)^2\frac{(1-\eps^2)^2(zd/n)^2}{2\cdot8^2\max\{d_i\}n}\leq -\eps^2(1-\eps)^2\frac{\eps^2 n^4 d^2/n^2}{128dn} \leq -\frac{\eps^4 nd}{512},\]
where in the second inequality we use that $z \geq \eps n^2$, and in the last we use that $1-\eps\geq1/2$.
Note also that we are still working here with the configurations that yield such graphs. We have not yet conditioned on these graphs being simple.
}
\[\mathbb{P}\left[Z'(M)\le (1-\eps)\frac{zd}{n}\right] \le 2e^{-\eps^4 nd/512}.\]
By definition, the same bound holds for $Z(F)$.
It now follows from \cref{prop: small subgraph} that\COMMENT{The following equation is of the form $\mathbb{P}[A \mid B] = \mathbb{P}[A, B]/\mathbb{P}[B] \le \mathbb{P}[A]/\mathbb{P}[B]$, where here $A$ is the property we are interested in, and $B$ is the property of being simple. We have
$\mathbb{P}[Z(F)\le (1-\eps)zd/n \mid R+F\text{ is simple}]\le  \mathbb{P}[Z(F) \le (1-\eps)zd/n]/\mathbb{P}[R+F\text{ is simple}]
\le (e^{3d^2})2e^{-\eps^4 nd/512}$.} 
\begin{equation}\label{eqnn1}
\mathbb{P}\left[Z(F)\le (1-\eps)\frac{zd}{n} \,\middle\vert\, R+F \text{ is simple}\right]\le 2e^{3d^2}e^{-\eps^4 nd/512}.
\end{equation}

By a similar argument we can show that \COMMENT{By \cref{lem: StochDom}, for each $a\in A$, the random variable $Y_a\sim\mathrm{Bin}(d_a, ((|Z_a|+1)/((1-\delta)n - 2)))$ stochastically dominates $e_F(a, V(Z_a))$. 
We have that $\mathbb{E}[Y_a] < (1+\eps^3)d|Z_a|/n$.
Since $|Z_a\cap E(F)|\leq e_F(a,V(Z_a))$, it follows that $\mathbb{E}[|Z_a \cap E(F)|] \le \mathbb{E}[Y_a] < (1+\eps^3)d|Z_a|/n$.
Therefore, $\mathbb{E}[Z(F)] \le (1+\eps^3)zd/n$.
As before, we have that 
\begin{align*}
    \mathbb{P}\left[Z'(M)\geq(1+\eps)\frac{zd}{n}\right]&=\mathbb{P}\left[Z'(M)\geq\frac{1+\eps}{1+\eps^3}(1+\eps^3)\frac{zd}{n}\right]\leq\mathbb{P}\left[Z'(M)\geq\frac{1+\eps}{1+\eps^3}\mathbb{E}[Z'(M)]\right]\\
    &\leq\mathbb{P}\left[|Z'(M)-\mathbb{E}[Z'(M)]|\geq\left(\frac{1+\eps}{1+\eps^3}-1\right)\mathbb{E}[Z'(M)]\right].
\end{align*}
We can now apply \cref{lem: mart}.
As a bound in the exponent we have
\[-\left(\frac{1+\eps}{1+\eps^3}-1\right)^2\frac{\mathbb{E}[Z'(M)]^2}{2\cdot8^2\max\{d_i\}n}\leq-\left(\frac{\eps(1-\eps^2)}{1+\eps^3}\right)^2\frac{(1-\eps^2)^2(zd/n)^2}{2\cdot8^2\max\{d_i\}n}\leq -\eps^2\frac{(1-\eps^2)^4}{(1+\eps^3)^2}\frac{\eps^2 n^4 d^2/n^2}{128dn} \leq -\frac{\eps^4 nd}{512},\]
where in the first inequality we have used the lower bound on $\mathbb{E}[Z'(M)]$ derived in the first part of the proof, in the second inequality we use that $z \geq \eps n^2$, and in the last we use that $(1-\eps^2)^4/(1+\eps^3)^2\geq1/4$.
Therefore,
\[\mathbb{P}\left[Z'(M)\ge (1+\eps)\frac{zd}{n}\right] \le 2e^{-\eps^4 nd/512},\]
and the same bound holds for $Z(F)$.
The final bound follows by \cref{prop: small subgraph}.}
\begin{equation}\label{eqnn2}
\mathbb{P}\left[Z(F)\ge (1+\eps)\frac{zd}{n} \,\middle\vert\, R+F \text{ is simple}\right] \le 2e^{3d^2}e^{-\eps^4 nd/512}.
\end{equation}
The result follows by combining \eqref{eqnn1} and \eqref{eqnn2}.
\end{proof}

\begin{lemma}\label{lem: spanset}
For every $0<\delta <10^{-5}$ there exists $D \in \mathbb{N}$ such that for any $D< d \le \log^2 n$ we have that a.a.s.~the random graph $G_{n,d}$ satisfies the following properties.
\begin{enumerate}[label=(\roman*)]
    \item\label{item: spanset1} For every $S\subseteq [n]$ with $\delta^2 d \leq |S| \leq 5\delta^2 n$, we have $e_{G_{n,d}}(S)\leq \delta d |S|/25$.
    \item\label{item: spanset2} For every $S\subseteq [n]$ with $5\delta^2 n \leq |S| \leq n/100$, we have $e_{G_{n,d}}(S) \leq d|S|/25$.
\end{enumerate}
\end{lemma}

\begin{proof}
Let $1/D \ll \delta$.
For any $D\leq d\leq\log^2n$, let $G\sim\mathcal{C}_{n,d}$.
For each $S \subseteq [n]$ such that $\delta^2 d \leq |S| \leq 5\delta^2 n$ and any multigraph $F$ on $[n]$, let $g(S,F)$ be the event that $e_{F}(S)\leq \delta d |S|/25$.
It follows by \cref{lem: StochDom} that the variable $e_G(S)$ is stochastically dominated by $Y\sim\mathrm{Bin}(d|S|,5|S|/(4n))$\COMMENT{By \cref{lem: StochDom}, it is dominated by $\mathrm{Bin}(d|S|,|S|/(n-2|S|))$, which has a lower expectation than what we claim in this range of $|S|$.}.
We denote by $\hat{\mathbb{P}}$ the probability measure associated with the configuration model and let $\mathbb{P}$ be the measure associated with the space of (simple) $d$-regular graphs.
Therefore, by \cref{lem: betaChernoff} we have \COMMENT{
$\hat{\mathbb{P}}\left [e_G(S) \geq \delta d|S|/25 \right] = \hat{\mathbb{P}} \left [e_G(S) \geq (4\delta n/(125|S|)) 5d|S|^2/(4n) \right] \leq (125e|S|/(4\delta n))^ { d\delta |S|/25} < (|S|/en)^{2|S|}$.
To see the last inequality, note that 
\[\left(\left(\frac{125e|S|}{4\delta n}\right)^{ d\delta /50}\frac{en}{|S|}\right)^{2|S|} = \left(\left(\frac{125e|S|}{4\delta n}\right)^{ d\delta /50 -1}\frac{125e|S|}{4\delta n}\frac{en}{|S|}\right)^{2|S|} \leq \left(\frac{125e^2}{4\delta}(1000\delta e)^{\delta d/50 -1}\right)^{2|S|} < 1.\]}
\[\hat{\mathbb{P}}[\overline {g(S,G)}]\leq\hat{\mathbb{P}}\left [e_G(S) \geq \delta d|S|/25 \right] = \hat{\mathbb{P}} \left [e_G(S) \geq (4\delta n/(125|S|)) 5d|S|^2/(4n) \right]  < (|S|/en)^{2|S|}. \]
To see this last inequality, note that 
\begin{align*}
    \left(\left(\frac{125e|S|}{4\delta n}\right)^{ d\delta /50}\frac{en}{|S|}\right)^{2|S|} &= \left(\left(\frac{125e|S|}{4\delta n}\right)^{ d\delta /50 -1}\frac{125e|S|}{4\delta n}\frac{en}{|S|}\right)^{2|S|}\\
    &\leq \left(\frac{125e^2}{4\delta}(1000\delta e)^{\delta d/50 -1}\right)^{2|S|} < 1.
\end{align*}
It follows by \cref{prop: small subgraph} that \COMMENT{Where the last inequality can be seen to follow since $(i/en)^{i}$ is decreasing in $i$: Let $f(i) = (i/en)^i = e^{i\log i - i\log n - i}$. We have that $f'(i) =e^{i\log i - i\log n - i} (\log i - \log n)$ where the term in the brackets is negative on this range. Hence the function is decreasing.
It follows then that 
\[\mathbb{P}\Big[\bigvee_{\substack{S \subseteq [n]\\\delta^2 d \leq |S| \leq 5\delta^2 n }} \overline{g(S,G_{n,d})}\Big]\leq e^{3d^2}5\delta^2 n\left(\frac{\delta^2d}{en}\right)^{\delta^2 d}=o(1).\]}
\begin{align*}
\mathbb{P}\Big[\bigvee_{\substack{S \subseteq [n]\\\delta^2 d \leq |S| \leq 5\delta^2 n }} \overline{g(S,G_{n,d})}\Big]&= \hat{\mathbb{P}}\Big[\bigvee_{\substack{S \subseteq [n]\\\delta^2 d \leq |S| \leq 5\delta^2 n }} \overline{g(S,G)}\mid G \text{ is simple}\Big] \\ 
&\le e^{3d^2}\sum_{\substack{S \subseteq [n]\\\delta^2 d \leq |S| \leq 5\delta^2 n }} \hat{\mathbb{P}}[\overline {g(S,G)}] \\
&\le e^{3d^2}\sum_{i=\delta^2 d}^{5\delta^2n} \binom{n}{i} \left(\frac{i}{en}\right)^{2i}=o(1).
\end{align*}
Thus, property \ref{item: spanset1} in the statement holds with probability $1-o(1)$.
Similarly, we can show that property \ref{item: spanset2} also holds with probability $1-o(1)$.
\end{proof}

\begin{proposition}\label{prop: small exp}
For every $0<\delta<10^{-5}$ there exists $D \in \mathbb{N}$ such that for any $D< d \le \log^2 n$ we have that a.a.s.~the random graph $G=G_{n,d}$ satisfies the following properties.
\begin{enumerate}[label=(\roman*)]
    \item\label{item: small exp1} Let $R \subseteq G$ be a spanning subgraph with $\delta(R) > \delta d$.
    Then, for every $S \subseteq [n]$ with $|S| \le \delta^2 n$, we have $|N_R(S)| \ge 3|S|$.
    \item\label{prop: bigset} For every $S, S' \subseteq [n]$ with $\delta^2 n \le |S| \le |S'| \le 3|S| \le 3n/400$, we have $e_{G}(S,S') \le d|S|/5$.
\end{enumerate}
\end{proposition}

\begin{proof}
Let $1/D \ll \delta $ and condition on the statement of \cref{lem: spanset} holding, which occurs a.a.s.
We first prove \ref{item: small exp1}.
For each $S \subseteq [n]$ such that $|S| < \delta^2 d$, the fact that every vertex has degree at least $\delta d$ ensures that $|N_{R}(S)|\geq \delta d>3\delta^2d$.
Now let $S\subseteq [n]$ with $\delta^2d\leq|S|\leq \delta^2 n$.
Suppose $|N_{R}(S)|< 3|S|$.
Let $Y\subseteq [n]$ be such that $|Y| = 3|S|$ and $N_{R}(S)\subseteq Y$.
We have by \cref{lem: spanset}\ref{item: spanset1} that 
\[4|S|\delta d/25 \ge e_G(S \cup Y) \ge e_R(S \cup Y) \ge e_R(S, Y)\ge |S|\delta d - e_{G}(S) > |S| \delta d/2,\]
a contradiction. 
The result follows.

In order to prove \ref{prop: bigset}, let $S\subseteq [n]$ with $\delta^2 n\le |S|\leq  n/400$.
Suppose there exists $S'\subseteq [n]$ with $|S| \le|S'| \le 3|S| $ and such that $e_{G}(S,S') > d|S|/5$.
We have by \cref{lem: spanset}\COMMENT{If $|S|<5\delta^2n$, use \ref{item: spanset1}, which gives a better bound; otherwise, use \ref{item: spanset2}.} that 
\[4|S| d/25  \ge e_{G}(S \cup S') \ge e_{G}(S, S') > d|S|/5,\]
a contradiction. 
The result follows.
\end{proof}

\begin{proposition}\label{prop: expexist}
For every $0 <\delta < 10^{-5}$ there exists $D\in \mathbb{N}$ such that for any $D < d \le \log^2 n$ we have that a.a.s.~the random graph $G=G_{n,d}$ has the following property.
Let $H \in \mathcal{H}_{1/2}(G)$ and let $G'\coloneqq G \setminus H$.
Then, there exists a spanning graph $R \subseteq G'$ such that $\Delta(R) < \delta d$ and, for every $S \subseteq [n]$ with $|S| \le n/400$, we have that $|N_R(S)| \ge 3|S|$.
\end{proposition}

\begin{proof}
Let $1/D \ll \delta$ and let $\hat{\delta} \coloneqq \delta/8$.
Condition on the event that the statements of \cref{lem: spanset,prop: small exp} hold with $\hat{\delta}$ playing the role of $\delta$, which happens a.a.s.
Suppose $G$ satisfies these events and $H \in \mathcal{H}_{1/2}(G)$, and let $G'\coloneqq G \setminus H$.
We now construct a suitable $R$ for this $G'$.
Consider a random subgraph $R$ of $G'$ where each edge is chosen independently and uniformly at random with probability $4\hat{\delta}$.
Consider the following events.
\begin{enumerate}[label=($\mathcal{G}$\arabic*)]
    \item\label{item: local} For all $v \in [n]$ we have $ \hat{\delta} d < d_R(v) < 8\hat{\delta} d$.
    \item\label{item: small edges} For every $S \subseteq [n]$ with $|S| \le n/400$, we have $|N_{R}(S)| \ge 3|S|$.
\end{enumerate}
Note that, if both \ref{item: local} and \ref{item: small edges} hold, then $R$ is a subgraph which satisfies the properties in the statement of the lemma.

For each $v \in [n]$, let $\mathcal{A}_v$ be the event that $d_R(v) \notin (\hat{\delta} d, 8 \hat{\delta} d)$.
By \cref{lem: Chernoff}, we have\COMMENT{
We have $2\hat{\delta} d \le \mathbb{E}[d_R(v)] \le 4\hat{\delta} d$ for all $v\in[n]$.
Therefore, by \cref{lem: Chernoff} we have that
\[\mathbb{P}[d_R(v)\leq\hat{\delta}d]\leq\mathbb{P}[d_R(v)\leq\mathbb{E}[d_R(v)]/2]\leq2e^{-\mathbb{E}[d_R(v)]/12}\leq2e^{-\hat{\delta} d/6}\]
and
\[\mathbb{P}[d_R(v)\geq8\hat{\delta}d]\leq\mathbb{P}[d_R(v)\geq6\hat{\delta}d]\leq\mathbb{P}[d_R(v)\geq3\mathbb{E}[d_R(v)]/2]\leq2e^{-\mathbb{E}[d_R(v)]/12}\leq2e^{-\hat{\delta} d/6}.\]
The bound follows by combining these two.} 
\[\mathbb{P}[\mathcal{A}_v] < 4e^{-\hat{\delta} d/6}\]
for all $v\in[n]$.
Observe that $G'$ is itself a dependency graph for $\{\mathcal{A}_v\}_{v\in[n]}$, and it has degree at most $d$.
By \cref{lem: LLL}, it follows that\COMMENT{By \cref{lem: LLL} we have that 
\[\mathbb{P}\Big[\bigwedge_{v \in [n]} \overline{\mathcal{A}_v}\Big] \ge (1-4ee^{-\hat{\delta} d/6})^n\ge (1-12e^{-\hat{\delta} d/6})^n.\]
In order to obtain the final bound observe that $12e^{-\hat{\delta} d/6}<1/2$ since $1/d\ll\delta$.}
\[\mathbb{P}[R\text{ satisfies \ref{item: local}}]=\mathbb{P}\Big[\bigwedge_{v \in [n]} \overline{\mathcal{A}_v}\Big] \ge (1-12e^{-\hat{\delta} d/6})^n \ge 2^{- n}.\]

Next, for $S, S' \subseteq [n]$, let $g(S, S')$ be the event that $N_R(S) \subseteq S'$.
Let $(\mathcal{G}3)$ be the event that for no pair of subsets $S, S' \subseteq [n]$ with $S' \subseteq N_{G'}(S)$ and $\hat{\delta}^2 n \le |S| \le |S'|\le 3|S| \le 3n/400$ the event $g(S,S')$ occurs.
We have by \cref{prop: small exp}\ref{prop: bigset} and \cref{lem: spanset} that
\[e_{G'}(S, [n] \setminus S')\geq d|S|/2 - e_{G'}(S, S')- e_{G'}(S) \ge d|S|/2 - d|S|/5 - d|S|/25 \geq d|S|/5.\] 
Therefore, we have
\[\mathbb{P}[g(S, S')] \le (1-4\hat{\delta})^{d |S|/5} \le e^{-4\hat{\delta}d|S|/5}\leq 2^{-4n}.\]
A union bound implies that $\mathbb{P}[R\text{ fails to satisfy ($\mathcal{G}$3)}] \le 2^{2n} 2^{-4 n} < 2^{-n}$.
Therefore, there exists an instance of $R$ which satisfies both \ref{item: local} and $(\mathcal{G}3)$ simultaneously.
Furthermore, since $R$ satisfies \ref{item: local}, it follows by \cref{prop: small exp}\ref{item: small exp1} that for every $S \subseteq [n]$ with $|S| \le \hat{\delta}^2 n$ we have that $|N_R(S)| \ge 3|S|$. 
Combining this with $(\mathcal{G}3)$ we see that $R$ also satisfies \ref{item: small edges}.
Thus, $R$ is a subgraph of the desired form.
\end{proof}

\begin{proposition}\label{prop: connexp}
For every $\eps >0$ there exists $D >0$ such that for any $D< d \le \log^2 n$ we have that a.a.s.~the random graph $G=G_{n,d}$ has the following property.
Let $H \in \mathcal{H}_{1/2 - \eps}(G)$ and let $G'\coloneqq G \setminus H$.
Let $R \subseteq G'$ be a spanning graph such that, for every $S \subseteq [n]$ with $|S| \le n/400$, we have $|N_R(S)| \ge 3|S|$.
Then, there exists a spanning $3$-expander $R' \subseteq G'$ with $e(R') \leq e(R) + 400$. 
\end{proposition}

\begin{proof}
Let $1/D \ll \eps$.
We are first going to prove that a.a.s.~$G'$ is connected.
Note that a.a.s., for any $A,B \subseteq [n]$ with $|A|=n/400$ and $|B|= (1/2-\eps/10)n$, we have $\sum_{a \in A}e_G(a,B) > (1/2-\eps/5)|A|d$.
Indeed, this follows by an application of \cref{prop: edges} with $R\coloneqq\varnothing$ and $Z_a$ being the star with centre $a$ whose leaves are all the vertices in $B\setminus\{a\}$.\COMMENT{Indeed, for each $a \in A$, let $Z_{a}$ be a star with centre $a$ and having every vertex in $B\setminus\{a\}$ as a leaf, and let $R$ be the empty graph.
Then, $z=\sum_{a \in A} |Z_a|=(|B|\pm1)|A|= (1/2-\eps/10)|A|n\pm|A|$, and hence, $|A|\leq z/((1/2-\eps/9)n)$.
By applying \cref{prop: edges} we get that 
\begin{align*}
    \mathbb{P} \left[\sum_{a\in A} e_G(a,B)\leq\left(\frac12-\frac\eps5\right)d|A|\right]&\leq\mathbb{P}\left[\sum_{a\in A} e_G(a,B)\leq\frac{1/2-\eps/5}{1/2-\eps/9}\frac{zd}{n} \right] \leq \mathbb{P}\left[\sum_{a\in A} e_G(a,B)\leq\left(1-\frac\eps6\right)\frac{zd}{n} \right]\\
    &\leq\mathbb{P}\left[\left|\sum_{a\in A} e_G(a,B)-\frac{zd}{n}\right|\geq\frac\eps6\frac{zd}{n} \right]\leq e^{-(\eps/60)^4dn}\ll 2^{-2n}.
\end{align*}
The second inequality follows since 
\[\frac{1/2-\eps/5}{1/2-\eps/9}=1-\frac{\eps/5-\eps/9}{1/2-\eps/9}\leq1-\frac{8\eps}{45}\leq1-\frac\eps6.\]
Finally, the claim in the proof follows by a union bound.}
We now claim that for any $A\subseteq [n]$ with $|A|\ge n/400$ we have that 
\begin{equation}\label{eq: expansionG'}
    |N_{G'}(A)| \ge (1/2 + \eps/10)n.
\end{equation}
To see this, note that if there exists $A\subseteq [n]$ with $|A|\ge n/400$ and  $|N_{G'}(A)| < (1/2 + \eps/10)n$ then we may take subsets $A'\subseteq A$ with $|A'| = n/400$ and $B\subseteq [n]$ with $|B|= (1/2- \eps/10)n$ such that $e_{G'}(A',B) = 0$.
However, we have already noted that for such $A'$ and $B$ we have that $\sum_{a \in A'}e_G(a,B) \ge (1/2-\eps/5)|A'|d$.
It follows that there exists $a \in A'$ with $e_G(a,B)> (1/2 - \eps/5)d$ and therefore $e_{G'}(a, B) >0$.
Thus, no such $A$ and $B$ exist.

In particular, \eqref{eq: expansionG'} implies that $G'$ is connected.
Indeed, assume that $G'$ is not connected and let $A\varsubsetneq[n]$ be a (connected) component of size $|A|\leq n/2$.
We must have that $|N_{G'}(A)|\leq|A|$, but \eqref{eq: expansionG'} and the statement hypotheses imply that $|N_{G'}(A)|>|A|$, a contradiction.

Finally, note that $R$ consists of at most $400$ components, since each connected component has order at least $n/400$.
Since $G'$ is connected, we may choose a set $E \subseteq E(G')$ with $|E| \le 400$ such that the graph $R' \coloneqq([n], E(R) \cup E)$ is connected, and thus is a spanning $3$-expander. 
\end{proof}

\begin{lemma}\label{thm: expmain}
For every $\eps >0$ and $0\leq\delta\leq10^{-5}$ there exists $D>0$ such that for any $D < d \le \log^2 n$ we have that a.a.s.~the random graph $G=G_{n,d}$ has the following property.
Let $H \in \mathcal{H}_{1/2-\eps}(G)$ and let $G'\coloneqq G \setminus H$.
Then, there exists a spanning $3$-expander $R \subseteq G'$ with $\Delta(R) < \delta d$.
\end{lemma}

\begin{proof}
Let $1/D \ll \delta, \eps$ and condition on the statements of \cref{prop: expexist,prop: connexp} both holding with $\delta/2$ playing the role of $\delta$, which happens a.a.s.
By \cref{prop: expexist} we may find a spanning subgraph $R' \subseteq G'$ with $\Delta(R') < \delta d/2$ and such that, for all $S \subseteq [n]$ with $|S| \le n/400$, we have $|N_{R'}(S)| \ge 3|S|$.
Then, by \cref{prop: connexp} we may find a spanning $3$-expander $R \subseteq G'$ with $\Delta(R) < \Delta(R') +400 < \delta d$.
\end{proof}


\section{Finding many boosters}\label{section: proof}

The following proposition provides an upper bound on the expected number of `thin' subgraphs that $G_{n,d}$ contains.

\begin{proposition}\label{prop: small subgraphs}
Let $1/n\ll1/d, \delta \ll 1$, where $n,d \in \mathbb{N}$, and let $G = G_{n,d}$.
Let $\mathcal{R}$ be a family of graphs on vertex set $[n]$ with $e(R) \leq \delta d n$ for all $R \in \mathcal{R}$.
Then, \[\sum_{R \in \mathcal{R}}\mathbb{P}[R \subseteq G] \le e^{2\delta d n \log(1/\delta)}.\]
\end{proposition}

\begin{proof}
For each $R\in\mathcal{R}$, let $X_R$ be an indicator random variable where $X_R(G)= 1$ if and only if $R \subseteq G$.
Let $X_{\mathcal{R}} \coloneqq \sum_{R \in \mathcal{R}}X_R$.
Then $\mathbb{E}[X_{\mathcal{R}}] =  \sum_{R \in \mathcal{R}}\mathbb{P}[R \subseteq G]$.
Moreover, note that we always have  $X_{\mathcal{R}} \leq \sum_{i=1}^{\delta d n} \binom{dn/2}{i} \le  e^{2\delta d n \log(1/\delta)}$ and, therefore, $ \sum_{R \in \mathcal{R}}\mathbb{P}[R \subseteq G]=\mathbb{E}[X_{\mathcal{R}}] \le  e^{2\delta d n \log(1/\delta)}$, as desired. \COMMENT{
$\sum_{i=1}^{\delta d n} \binom{dn/2}{i} \le (\delta d n)(e/(2\delta))^{\delta d n} = e^{\log(\delta dn)+\delta d n(1 - \log2 + \log(1/\delta))} \le  e^{2\delta d n \log(1/\delta)}$.}
\end{proof}

The following result can easily be proved using ``P\'osa rotations'' (see e.g.~\cite{Kri16}). \COMMENT{
Suppose we consider every sequence of rotations of $P$ to obtain a set of endpoints $A$ of maximum size and suppose that $|A| < n/400$.
Since $R$ is a $3$-expander we have that $N_R(A) \ge 3|A|$.
It follows that there exists a vertex $a \in A$ with a neighbour $x \in V(P)$ such that none of $x, x^{+}$ or $x^{-}$ are elements in $A$, where $x^{+}$ and $x^{-}$ are the successor and predecessor of $x$ on $P$. (Note that each $a \in A$ can `block' 3 potential vertices: itself, its predecessor and its successor. 
This totals $3|A|$ vertices.
However, the endpoint $a' \in A$ of $P$ can only block itself and its successor. Therefore at most $3|A|-1$ are blocked and the claim follows.)
So $xa$ is an edge.
Consider the path $P_a$.
It follows that we can rotate $P_a$ to obtain a path $P_{x'}$, with $x'\in\{x^+,x^-\}$. 
Thus, $A$ was not maximum.
By the definition of a $3$-expander we can continue like this until $|A| \ge n/400$ and the result follows.}

\begin{lemma}\label{lem: posa}
Let $R$ be a $3$-expander and let $P$ be a longest path in $R$, with endpoint $v$.
Then, there exists a set $A \subseteq V(P)$ with $|A| > n/10^4$ such that for each $a \in A$ there exists a path $P_a$ in $R$ with endpoints $v$ and $a$, and such that $V(P_a) = V(P)$.
\end{lemma}

\begin{definition}[Booster]
Let $H$ be a graph and let $E \subseteq V(H)^{(2)}$.
Let $F\coloneqq(V(H),E)$.
We call $E$ a \emph{booster for $H$} if the graph $H+F$ contains a longer path than $H$ does, or if $H+F$ is Hamiltonian. 
\end{definition}

We will often be interested in the case where $E$ consists of a single edge  $e\notin E(H)$.
In this case we refer to $e$ as a \emph{booster} for $H$.

Given any path $P$ with endpoints $u$ and $v$, assume an orientation on its edges (say, from $u$ to $v$).
Given any vertex $x\in V(P)\setminus\{v\}$, we call the vertex that follows $x$ in this orientation its \emph{successor}, and we denote this by $\mathit{suc}_P(x)$.

\begin{lemma}\label{lem: booster}
For all $0 < \eps < 1/10^5$ there exist $\delta, D >0$ such that for $D \le d \le \log^2 n$ the random graph $G = G_{n,d}$ a.a.s.~satisfies the following.

Let $H \in \mathcal{H}_{1/2 - \eps}(G)$ and let $G'\coloneqq G\setminus H$.
Let $R\subseteq G'$ be a spanning $3$-expander with $\Delta(R) \le 2\delta d$, and let $S \subseteq [n]$ with $|S|\le \delta n$.
Then, there exists a set $V_R \subseteq [n]$ with $|V_R|\geq n/10^4$ with the following property: 
for each $v \in V_R$, there exists a set $U_v \subseteq [n]$ with $|U_v| \ge (1/2 + \eps/8)n$ such that, for each $u \in U_v$, there exists a set $E_{v,u}$ as follows:
\begin{enumerate}[label=(\alph*)]
    \item\label{itm: a} $E_{v,u} \subseteq E((G'\setminus R)[[n]\setminus S])$ with $|E_{v,u}| \ge 50/(\eps \delta)$,
    \item\label{itm: b} $\{uv, e\}$ is a booster for $R$ for every $e \in E_{v,u}$,
    \item\label{itm: c} $E_{v,u_1}\cap E_{v,u_2}=\varnothing$ for all $u_1\neq u_2$.
\end{enumerate}
\end{lemma}

\begin{proof}
Let $1/D \ll \delta \ll \eps < 1/10^5$.
Let $\mathcal{R}$ be the set of all $n$-vertex $3$-expander graphs $R$ on $[n]$ with $\Delta(R) \le 2\delta d$.
It follows by \cref{lem: posa} that for each $R\in\mathcal{R}$ there exists a set $V_R\subseteq[n]$ of size $|V_R|\geq n/10^4$ such that for every $v \in V_R$ there exists a longest path in $R$ terminating at $v$.

For each $R\in \mathcal{R}$, $v \in V_R$ and $S \subseteq [n]$ with $|S| \leq \delta n$, let $f(R,S,v)$ be the event that, for every $H \in \mathcal{H}_{1/2 - \eps}(G)$ such that $R\subseteq G'$, there exists a set of vertices $U_v \subseteq [n]$ with $|U_v| \ge (1/2 + \eps/8)n$ and such that for each $u \in U_v$ there exists a set $E_{v,u}$ satisfying \ref{itm: a}--\ref{itm: c}.
With this definition, the probability $p^{*}$ that the assertion in the lemma fails is bounded by
\begin{equation}\label{p} 
p^{*}\leq\sum_{S\subseteq[n]:|S|\leq\delta n}\sum_{R \in \mathcal{R}} \sum_{v \in V_R} \mathbb{P}[\overline{f(R,S,v)} \mid R \subseteq G]\,\mathbb{P}[R \subseteq G].
\end{equation}

For fixed $R \in \mathcal{R}$, $v \in V_R$ and $S \subseteq [n]$ with $|S| \leq \delta n$, we shall now estimate $\mathbb{P}[\overline{f(R,S,v)} \mid R \subseteq G]$.
Let $P$ be a longest path in $R$ with endpoint $v$.
As $R$ is a 3-expander, by \cref{lem: posa} there must exist a set $A \subseteq V(P)\setminus S$ of size $|A|=\eps n/20$ such that, for each $a \in A$, there is a longest path $P_a$ in $R$ starting at $v$ and ending at $a$ with $V(P_a)=V(P)$ (if there is more than one such path, fix one arbitrarily).
Assume that each $P_a$ is oriented from $v$ to $a$.
Let $B\coloneqq [n] \setminus (A\cup S\cup \{v\})$.
For each $u \in B \cap V(P)$, let $X_u\coloneqq\{ab:a\in A,b\in B,u=\mathit{suc}_{P_a}(b)\}$.
Observe that $\{uv, ab\}$ is a booster for $R$ for any $ab\in X_u$\COMMENT{Here we use that $R$ is connected.}. 
Clearly, $|X_u| \le |A|$ and $X_u \cap X_{u'} = \varnothing$ for all distinct $u, u' \in B\cap V(P)$.
Furthermore, for each $u \in B\setminus V(P)$, let $X_u\coloneqq\{au:a\in A\}$. 
Note that $au\in X_u$ is a booster since its inclusion would result in a longer path in $R$.
We shall now show that, for most vertices $u\in B$, there is a `large' set of boosters, that is, $X_u$ is `large'.
We will then use this to show that many of these boosters must lie in $G'\setminus R$.

For every $a \in A$, there are at least $|V(P)| - 2|A| - 2|S|-2$ vertices $b \in V(P)$ such that neither $b$ nor its successor on $P_a$ belong to $A\cup S\cup \{v\}$.
It follows that $|\bigcup_{u \in B\cap V(P)} X_u| \ge |A|(|V(P)| - 2|A|- 2|S|-2)$.
We also have that $|\bigcup_{u \in B \setminus V(P)} X_u| = |A|(n - |V(P)\cup S|)$.
Therefore, the following holds: 
\begin{align*}
    \bigg|\bigcup_{u \in B}X_u\bigg| &\geq |A|(|V(P)|- 2|A| -2|S|-2) + |A|(n - |V(P)\cup S|) \\
    &\geq |A| (n - 2|A| - 3|S|-2) \geq (1-\eps/9)|A|n.
\end{align*}

For each $u \in B$, let $Y_u \coloneqq X_u \setminus E(R)$.
It follows that 
\[\bigg|\bigcup_{u \in B}Y_u\bigg| \ge (1-\eps/9)|A|n -e(R) \ge (1-\eps/8)|A|n.\]
For each $a \in A$, let $Z_a$ be the set of edges in $\bigcup_{u \in B} Y_u$ with $a$ as an endpoint. 
It is easy to see that $\sum_{a\in A}|Z_a|=|\bigcup_{a \in A} Z_a| = |\bigcup_{u \in B}Y_u|\ge (1-\eps/8)|A|n$.
Consider now the following events:
\begin{enumerate}[label=$\mathcal{F}$\arabic*:]
\item\label{item: propF1} $|\bigcup_{a \in A}(Z_a \cap E(G))| \ge (1-\eps/4)|A|d$.
\item\label{item: propF2} For any $U \subseteq B$ with $|U| \le (1/2 + \eps/8)n$ we have $|\bigcup_{u \in U}Y_u \cap E(G)| < (1/2 + \eps/4)|A|d$. 
\end{enumerate}
From two applications of \cref{prop: edges} we obtain that $\mathbb{P}[\mathcal{F}_1 \wedge \mathcal{F}_2 \mid R \subseteq G] \ge 1 - e^{-(\eps/500)^4dn}$.
To finish the proof we must show that if $\mathcal{F}_1 \wedge\mathcal{F}_2$ holds, then $f(R,S,v)$ also holds.
Consider any $G\in\mathcal{G}_{n,d}$ which satisfies both $\mathcal{F}_1$ and $\mathcal{F}_2$ and such that $R\subseteq G$.
Fix any $H\in\mathcal{H}_{1/2-\eps}(G)$ such that $R\subseteq G'=G\setminus H$.
For each $u \in B$, let $E_u \coloneqq Y_u \cap E(G')$.
As we have seen above, for each $e \in E_u$, the set $\{uv, e\}$ is a booster for $R$.
Furthermore, none of the endvertices of $e$ lies in $S$, by construction.
Let $U\subseteq B$ be the set of vertices $u\in B$ for which $|E_u| \ge 50/(\eps \delta)$.
Observe that, by $\mathcal{F}_2$, if $|\bigcup_{u \in U}Y_u \cap E(G)| \geq (1/2 + \eps/4)|A|d$, then $|U| \ge (1/2 + \eps/8)n$.
But\COMMENT{In the fourth line we are using $|A|=\eps n/20=50\eps n/1000$.}
\begin{align*}
\bigg|\bigcup_{u \in U}Y_u \cap E(G)\bigg|\geq\bigg|\bigcup_{u \in U}E_u\bigg| & = \bigg|\bigcup_{u \in B}Y_u \cap E(G)\bigg| - \bigg|\bigcup_{u\in B}Y_u \cap E(H)\bigg| - \sum_{u \in B\setminus U}|E_u|\\
 &\ge \bigg|\bigcup_{a \in A}Z_a \cap E(G)\bigg| - \bigg|\bigcup_{a \in A}Z_a \cap E(H)\bigg| - \frac{50}{\eps\delta}|B\setminus U|\\
 &\stackrel{\mathclap{(\mathcal{F}1)}}{\ge}\left(1- \frac\eps4\right)|A|d - \sum_{a \in A}\left(\frac12 - \eps\right)d_G(a) - \frac{50}{\eps\delta}n\\
 &\ge\left(1- \frac\eps4\right)|A|d - \left(\frac12 - \eps\right)|A|d - 10^3\frac{|A|d}{\eps^2d\delta}\\
 &\ge \left(\frac12 + \frac\eps4\right)|A|d.
\end{align*}
Hence, by $\mathcal{F}_2$ we have that $|U| \ge (1/2 + \eps/8)n$, as we wanted to show. 
Since $H$ was arbitrary, it follows that $f(R,S,v)$ holds.
Thus,
\[\mathbb{P}[f(R,S,v)\mid R\subseteq G]\geq\mathbb{P}[\mathcal{F}_1\wedge\mathcal{F}_2\mid R\subseteq G]\geq1-e^{-(\eps/500)^4dn}.\]

We can now use this bound in equation \eqref{p} to obtain
\[p^* \le 2^nne^{-(\eps/500)^4dn} \sum_{R \in \mathcal{R}} \mathbb{P}[R \subseteq G]\leq 2^nn e^{-(\eps/500)^4dn} e^{2\delta dn \log (1/\delta)} = o(1),\]
where the second inequality follows from \cref{prop: small subgraphs}.
This shows the statement in the lemma holds a.a.s.
\end{proof}

\begin{definition}\label{def: help}
Given graphs $H$ and $H'$ with $V(H) = V(H') = V$ and $E(H) \cap E(H') = \varnothing$, we say $H$ has \emph{$\eps$-many boosters with help from $H'$} if there are at least $\eps|V|$ vertices $v \in V$ for which there exists a set $U_v \subseteq V\setminus \{v\}$ of size at least $(1/2 + \eps)|V|$ with the property that for every $u \in U_v$ there exists $e \in E(H')$ so that $\{uv, e\}$ is a booster for $H$. 
We call $uv$ the \emph{primary edge} and we call $e$ the \emph{secondary edge}.
\end{definition}

\begin{corollary}\label{cor: boost}
For all $0 <\eps < 1/10^5 $ there exist $\delta,D >0$ such that for $D \le d \le \log^2 n$ the random graph $G = G_{n,d}$ a.a.s.~satisfies the following.

Let $H \in \mathcal{H}_{1/2-\eps}(G)$ and let $G'\coloneqq G \setminus H$.
Let $R\subseteq G'$ be a spanning $3$-expander with $\Delta(R)\le 2\delta d$, and let $S \subseteq [n]$ with $|S| \leq \delta n$.
Then, there exists some subgraph $F \subseteq G'\setminus R$ satisfying $\Delta(F) \le 2\delta d$, such that $R$ has $\eps/16$-many boosters with help from $F$, with the property that the set of secondary edges is vertex-disjoint from $S$.
\end{corollary}

\begin{proof}
Let $1/D \ll \delta \ll \eps < 1/{10}^5$. 
Condition on the event that $G$ satisfies all the properties in the statement of \cref{lem: booster}, which happens a.a.s.
Let $H, G', R, S$ be as in the statement of \cref{cor: boost}.
By \cref{lem: booster}, we may find a set $V_R \subseteq [n]$ of size $|V_R|\geq n/10^4$ such that, for each $v\in V_R$, there exists a set $U_v \subseteq [n]$ with $|U_v| \ge (1/2 + \eps/8)n$ such that, for each $u\in U_v$, there exists a set $E_{v,u} \subseteq E((G'\setminus R)[[n]\setminus S])$ with $|E_{v,u}| \ge 50/(\eps \delta)$ and such that, for every $e \in E_{v,u}$, $\{uv, e\}$ is a booster for $R$, and such that $E_{v,u_1}\cap E_{v,u_2}=\varnothing$ for all $u_1\neq u_2$.
Note that each such edge $e$ is vertex-disjoint from $S$, by construction.

Let $F$ be a random subgraph of $G'\setminus R$ where every edge in $G'\setminus R$ is chosen independently at random with probability $\delta/2$.
For each $v \in V_R$, let $U'_v \subseteq U_v$ be the set of vertices $u\in U_v$ for which $E_{v,u} \cap E(F) \neq \varnothing$.
For every $u \in U_v$, we have that 
\[\mathbb{P}[u \notin U'_{v}] \le (1-\delta/2)^{50/(\eps \delta)} \le e^{-25/\eps} \le \eps/32.\]
Let $\mathcal{A}$ be the event that $|U'_v| \ge (1/2 + \eps/16)n$ for every $v \in V_R$.
Since for different $u \in U_v$ the sets $E_{v,u}$ are disjoint, by \cref{lem: Chernoff} we have \COMMENT{
Since $|U_v| \ge (1/2 + \eps/8)n$, we have $(1/2 + \eps/16)n \le (1/2 + \eps/16)(1/2 + \eps/8)^{-1}|U_v| \le (1-\eps/32)^2|U_v|$.
As $\mathbb{E}[|U_v'|]\geq(1-\eps/32)|U_v|$, it follows that $(1/2 + \eps/16)n \le (1-\eps/32)\mathbb{E}[|U_v'|]$.
By \cref{lem: Chernoff} we conclude that $\mathbb{P}[|U'_{v}| \le (1- \eps/32)(1- \eps/32)|U_v|] \le \mathbb{P}[|U'_{v}| \le (1- \eps/32)\mathbb{E}[|U'_v|]]\le 2e^{-(\eps/32)^2\mathbb{E}[|U'_v|]/3} \le  2e^{-(\eps/32)^2(1-\eps/32)|U_v|/3}\leq e^{-\eps^2n/10^6}$.}
\[\mathbb{P}[|U'_v| \le (1/2 + \eps/16)n] \le \mathbb{P}[|U'_{v}| \le (1- \eps/16)|U_v|] \leq e^{-\eps^2n/10^6}\]
for each $v\in V_R$.
Therefore, 
\[\mathbb{P}[\overline{\mathcal{A}}] \le ne^{-\eps^2 n/10^6} \le e^{-\eps^{3}n}.\]

Now, let $\mathcal{B}$ be the event that $\Delta(F) \le 2\delta d$.
For each $v \in [n]$, let $\mathcal{B}_v$ be the event that $d_F(v) > 2\delta d$.
By \cref{lem: betaChernoff}, we have\COMMENT{
We have $\delta d/8 \le \mathbb{E}[d_F(v)] \le \delta d/2$.
Therefore, by \cref{lem: betaChernoff} we have that \[\mathbb{P}[d_F(v) \ge 2\delta d]\leq\mathbb{P}[d_F(v) \ge 4\mathbb{E}[d_F(v)]] \leq (e/4)^{4\mathbb{E}[d_F(v)]}\leq e^{(1-\log4)\delta d/2}< e^{-\delta d/8}.\]}
\[\mathbb{P}[B_v] < e^{-\delta d/8}\]
for all $v\in[n]$.
Now observe that $G'\setminus R$ is itself a dependency graph for $\{\mathcal{B}_v\}_{v\in [n]}$, and every vertex in this graph has degree at most $d$.
It follows by \cref{lem: LLL} that \COMMENT{
Where we note that $e^{1-\delta d/8} \le \eps^4/2$ and $1-x \ge e^{-2x}$ for $x \le \eps^4/2$.}
\[\mathbb{P}[\mathcal{B}] = \mathbb{P}\left[\bigwedge_{v \in [n]} \overline{\mathcal{B}_v}\right] \ge (1-e^{1-\delta d/8})^n \ge  e^{-\eps^{4} n} > \mathbb{P}[\overline{\mathcal{A}}].\]
Therefore, the probability both events $\mathcal{A}$ and $\mathcal{B}$ occur is strictly positive, implying that there exists some $F \subseteq G'\setminus R$ satisfying the required properties.
\end{proof}

We have now shown that a.a.s.~if the random graph $G_{n,d}$ contains a sparse $3$-expander subgraph $R$ after deleting some $H\in \mathcal{H}_{1/2-\eps}(G_{n,d})$, then $G'=G_{n,d}\setminus H$ must also have a sparse subgraph $F$ with the property that $R$ has `many' boosters with help from $F$. 
Our next  goal is to prove that some primary edge of these boosters must actually be present in $G'$.

\begin{lemma}\label{lem: rand}
For all $0 < \eps < 1/10^5$ there exist $\delta, D>0$ such that for $D \le d \le \log^2n$ the random graph $G = G_{n,d}$ satisfies the following a.a.s.
Let $S \subseteq [n]$ with $|S| \le \delta n$ and let $R, F \subseteq G$ be two spanning edge-disjoint subgraphs such that
\begin{enumerate}[label=(P\arabic*)]
    \item\label{item: lastlemma1}  $\Delta(R), \Delta(F) \le 2\delta d$,
    \item\label{item: lastlemma3} $R$ has $\eps$-many boosters with help from $F$, such that every secondary edge is vertex-disjoint from $S$.
\end{enumerate}
Then, for any $H \in \mathcal{H}_{1/2}(G)$, the graph $G' \coloneqq G \setminus H$ contains an edge $e$ for which there exists some edge $ e' \in E(F)$ with the property that $\{e, e'\}$ is a booster for $R$, and such that $V(\{e, e'\}) \cap S = \varnothing$. 
\end{lemma}

\begin{proof}
Let $1/D \ll \delta \ll \eps< 1/10^5$.
Let $\mathcal{P}$ be the set of all triples $(R,F,S)$ where $R$ and $F$ are edge-disjoint graphs on $[n]$ which satisfy \ref{item: lastlemma1} and \ref{item: lastlemma3} and $S\subseteq[n]$ with $|S| \le \delta n$.

Fix a triple $(R, F, S)\in \mathcal{P}$. 
For every $x \in [n]$, let $V_x$ be the set of vertices $v \in [n] \setminus (S \cup \{x\})$ for which there exists some edge $e \in E(F)$ such that none of the endvertices of $e$ lies in $S$ and $\{xv, e\}$ is a booster for $R$.
Let $X' \coloneqq \{ x \in [n]: |V_x| \ge (1/2 + 3\eps/4)n\}$.
By assumption on the triple $(R,F,S)$ and using \cref{def: help}, we must have that $|X'| \ge \eps n$.
\COMMENT{For each $x\in[n]$, let $V'_x$ be the set of vertices $v \in [n] \setminus \{x\}$ for which there exists some edge $e \in E(F\setminus S)$ such that $\{xv, e\}$ is a booster for $R$.
We have by \ref{item: lastlemma3} and \cref{def: help} that $|V'_x| \ge (1/2 + \eps)n$ for at least $\eps n$ values of $x$.
But for each of these values of $x$ we have $|V_x|\geq|V_x'|-|S|\geq(1/2+3\eps/4)n$.}

Let $X\coloneqq X'\setminus S$.
Let $f(R,F,S)$ be the event that $\sum_{x \in X}e_{G \setminus (R + F)}(x, V_x) \ge (1 + \eps) d|X|/2$.
It follows by \cref{prop: edges} that\COMMENT{For each $x\in X$, let $Z_x$ be a star with centre $x$ and edges to all those vertices $v\in V_x$ for which $xv\notin E(R+F)$.
We have that $z \ge |X|\min\{|V_x|-d_{R+F}(x)\} \ge (\eps n - \delta n)(n/2) > \eps n^2/3$.
Apply \cref{prop: edges} with $\eps/3$ playing the role of $\eps$.
We get with probability at least $1 - e^{-(\eps/30)^4dn}$ (conditioned on $R+F\subseteq G$) that $\sum_{x \in X}e_{G\setminus(R+F)}(x, V_x) \ge (1-\eps/3)(d/n)|X|\min_{x\in X}\{|V_x|-d_{R+F}(x)\} > (1 + \eps) d|X|/2$.
}
\begin{equation}\label{comp}
\mathbb{P}[\overline{f(R,F,S)}\mid R+F \subseteq G] \le e^{-(\eps/30)^4dn}.
\end{equation}
It follows that the probability that $(R+F\subseteq G)\wedge\overline{f(R,F,S)}$ for some triple $(R,F,S) \in \mathcal{P}$ is at most \COMMENT{I've combined $R$ and $F$ below into a single graph since they live on the same vertex set and are edge disjoint.}
\begin{align*}
    \sum_{(R,F,S) \in \mathcal{P}} \mathbb{P}[\overline{f(R,F,S)}\mid R+F \subseteq G]\,\mathbb{P}[R+F \subseteq G] &\stackrel{\mathclap{\eqref{comp}}}{\le} 2^n e^{-(\eps/30)^4dn} \sum_{\substack{K \subseteq K_n\\ e(K) \leq 2\delta d n}} \mathbb{P}[K \subseteq G]\\
    & \le 2^n e^{-(\eps/30)^4dn} e^{4\delta dn \log 1/(2\delta)} \leq e^{-(\eps/50)^4dn},
\end{align*}
where the second inequality follows by \cref{prop: small subgraphs}.

We conclude that a.a.s.~for all $(R,F,S) \in \mathcal{P}$ with $R+F \subseteq G$ we have
\[\sum_{x \in X}\left(e_{G \setminus (R + F)}(x,V_x) - \frac12d_{G}(x)\right) \ge \left(1 + \eps\right)\frac{d|X|}{2} - \frac{d|X|}{2} > 0.\]
Hence, there must exist some $x \in X$ with $e_{G\setminus(R+F)}(x, V_x) > d/2$. 
Therefore, for any $H\in \mathcal{H}_{1/2}(G)$, there is some vertex $x\in X$ such that $e_{G'}(x, V_x) \ge e_G(x,V_x) - d_{H}(x) > 0$. 
That is, there must be some $v \in N_{G'}(x) \cap V_x$.
By the definition of $V_x$, there is some $e \in E(F)$ such that $\{xv,e\}$ is a booster for $R$.
Furthermore, by construction, we have $V(\{xv, e\}) \cap S = \varnothing$, and this completes the proof of the lemma.
\end{proof}

Armed with the previous lemmas, we are now in a position to complete the proof of \cref{thm: main}.

\begin{proof}[Proof of \cref{thm: main}]
Let $1/D \ll \delta \ll \eps < 1/10^5$ be such that \cref{cor: boost} holds for $\eps$, \cref{lem: rand} holds for $\eps/16$ and \cref{thm: expmain} holds for $\eps$.
Condition on each of these holding.

Let $H \in \mathcal{H}_{1/2 -\eps}(G)$ and let $G'\coloneqq G\setminus H$.
By \cref{thm: expmain}, there exists a subgraph $R \subseteq G'$ which is a spanning $3$-expander with $\Delta(R) \le \delta d$.

Let $R_0 \coloneqq R$.
We now proceed recursively as follows: for each $i\in[n]$, choose $e_{i,1}, e_{i,2} \in E(G')$ such that $\{e_{i,1}, e_{i,2}\}$ is a booster for $R_{i-1}$, and let $R_i \coloneqq R_{i-1} + e_{i,1} + e_{i,2}$.
In order to show that there exist such $e_{i,1}, e_{i,2}$ for all $i\in[n]$, consider the following.
Assume that $R_{i-1}$ satisfies $\Delta(R_{i-1}) \leq 2\delta d$.
Let $S_i \subseteq [n]$ be the set of vertices $v \in [n]$ with $d_{R_{i-1}}(v) \ge 2\delta d-1$.
For all $i \in [n]$ we have $|S_i| \leq 2e(R_{i-1}\setminus R_0)/(\delta d-1)\le 4n/(\delta d-1) < \delta n$. 
By applying \cref{cor: boost} with $S_i$, $R_{i-1}$ playing the roles of $S$ and $R$, respectively, there exists some subgraph $F_i \subseteq G' \setminus R_{i-1}$ such that $R_{i-1}$ has $(\eps/16)$-many boosters with help from $F_i$, where each secondary edge is vertex-disjoint from $S_i$.
Furthermore, we have $\Delta(F_i) \le 2\delta d$.
Therefore, by applying \cref{lem: rand}, there are some $e_{i,1}, e_{i,2} \in E(G')$ such that $\{e_{i,1}, e_{i,2}\}$ is a booster for $R_{i-1}$ and where $V(\{e_{i,1}, e_{i,2}\})\cap S_i = \varnothing$.
It follows that $\Delta(R_{i}) \leq 2\delta d$.

By the end of this process, we have added $n$ boosters to $R$ to obtain $R_n \subseteq G'$. 
Therefore $R_n$, and hence $G'$, is Hamiltonian.
\end{proof}


\section{Graphs of small degree with low resilience}\label{section: counter}

In this section we prove \cref{thm: counter}.
For this, we will require a crude bound on the number of edges spanned by any set of $n/2$ vertices in $G_{n,d}$.
To achieve this, we shall make use of the following result, which follows from a theorem of \citet{MacKay87} (see e.g.~\cite{SurvWormald}). 
We denote by $\alpha(G)$ the size of a maximum independent set in $G$.\COMMENT{
We have the following theorem due to McKay: A.a.s.~we have that $\alpha(G_{n,3})\leq 0.46n$.
We note that for any fixed $d$, the following two models are contiguous namely, $G_{n,d}$ and $G_{n,(d-1)} \oplus G_{n,1}$,~see e.g. the survey~\cite{SurvWormald}. 
This fact implies that for any $d\geq 4$, a.a.s.~$\alpha(G_{n,d})\leq 0.46n$ and therefore the lemma follows from the result of McKay. }

\begin{theorem}\label{thm: McKaydind}
For every fixed $d \ge 3$, a.a.s.~we have that $\alpha(G_{n,d})\leq 0.46n$. 
\end{theorem}

\begin{lemma}\label{prop: edgest}
For every fixed $d\geq3$, a.a.s.~we have that $e_{G_{n,d}}(A) > n/100$ for all $A \subseteq [n]$ with $|A| = \lfloor n/2 \rfloor$.
\end{lemma}

\begin{proof}
By \cref{thm: McKaydind}, every set of size $n/2$ must span at least $n/100$ edges, as otherwise it would contain an independent set of size $n/2-n/50>0.46n$.  
\end{proof}

Alternatively, this lemma can be proved directly using a switching argument.

In order to prove \cref{thm: counter} we will use a switching argument.
Given a graph $G\in\mathcal{G}_{n,d}$ and any integer $\ell\in[d]$, let $u\in[n]$ and let $\Lambda_{u,\ell}^+=(e_1,\ldots,e_\ell,f_1,\ldots,f_\ell)$ be an ordered set of $2\ell$ edges from $E(G)$ such that $e_i=uv_i$ with $v_i\neq v_j$ for all $i\neq j$, and $\{f_i:i\in[\ell]\}$ is a set of independent edges such that, for each $i\in[\ell]$, the distance between $f_i$ and $e_i$ is at least 2\COMMENT{Note that this implies that each $f_i$ is vertex-disjoint from the star consisting of $e_1,\ldots,e_\ell$.}.
We call $\Lambda_{u,\ell}^+$ a \emph{$(u,\ell)$-switching configuration}.
For each $i\in[\ell]$, choose an orientation of $f_i$ and write $f_i=x_iy_i$, where $f_i$ is oriented from $x_i$ to $y_i$.
Let $\lambda_{u,\ell}^+\coloneqq\{e_1,\ldots,e_\ell,f_1,\ldots,f_\ell\}$, $\Lambda_{u,\ell}^-\coloneqq(uy_1,\ldots,uy_\ell,x_1v_1,\ldots,x_\ell v_\ell)$ and $\lambda_{u,\ell}^-\coloneqq\{uy_1,\ldots,uy_\ell,x_1v_1,\ldots,x_\ell v_\ell\}$.
We say that the graph $G'\coloneqq([n],(E(G)\setminus\lambda_{u,\ell}^+)\cup\lambda_{u,\ell}^-)\in\mathcal{G}_{n,d}$ is obtained from $G$ by a \emph{$u$-switching of type $\ell$}.
Observe that, given such a setting, we also have that $G$ is obtained from $G'$ by a $u$-switching of type $\ell$, interchanging the roles of $\Lambda_{u,\ell}^+$ and $\Lambda_{u,\ell}^-$.

\begin{proof}[Proof of \cref{thm: counter}]
Fix any odd $d>2$.
Let $\hat{\mathcal{G}}_{n,d} \subseteq \mathcal{G}_{n,d}$ be the collection of graphs for which the statement of \cref{prop: edgest} holds.
We have by \cref{prop: edgest} that $|\hat{\mathcal{G}}_{n,d}| = (1-o(1))|\mathcal{G}_{n,d}|$.
Let $\mathcal{G}_{n,d}'\subseteq\hat{\mathcal{G}}_{n,d}$ be the collection of all graphs $G\in\hat{\mathcal{G}}_{n,d}$ which are not $(d-1)/2$-resilient with respect to Hamiltonicity.
Let $p\coloneqq|{\mathcal{G}}_{n,d}'|/|\hat{\mathcal{G}}_{n,d}|$.
We will prove that $p$ is bounded from below by a positive constant which does not depend on $n$. 

For each $G \in \hat{\mathcal{G}}_{n,d}$, consider a maximum cut $M$ with parts $A_M$ and $B_M$, where $|A_M|\le |B_M|$ (thus $M=E_G(A_M,B_M)$).
By abusing notation, we also use $M$ to denote the bipartite graph $G[A_M,B_M]$.
Note that for all $x \in A_M$ we have $d_{M}(x) > d/2$, as otherwise we could move $x$ from $A_G$ to $B_G$ to obtain a larger cut; similarly, $d_{M}(y) > d/2$ for all $y\in B_M$.

Given $G\in\hat{\mathcal{G}}_{n,d}$, suppose there exists a maximum cut $M$ for $G$ such that $|A_M|<|B_M|$.
Let $H\coloneqq([n],E_G(A_M)\cup E_G(B_M))$.
It is then clear that $M=G\setminus H$ is not Hamiltonian, as it is an unbalanced bipartite graph.
Furthermore, we have that $\Delta(H)\leq(d-1)/2$\COMMENT{The degree of each vertex $v\in A_M$ in $H$ is the number of edges this vertex had into $A_M$ in $G$. By the above observation, this is $d-d_M(v)\leq d-(d+1)/2=d/2-1/2$. Something similar works for $v\in B_M$.}, so we conclude that $G$ is not $(d-1)/2$-resilient with respect to Hamiltonicity and, thus, $G\in\mathcal{G}_{n,d}'$.
(Below we will use that the same conclusion holds if there is \emph{any} cut $M$ of $G$ such that $|A_M|<|B_M|$, $d_{M}(x) > d/2$ for all $x \in A_M$, and $d_{M}(y) > d/2$ for all $y\in B_M$.)
Therefore, for every $G \in \hat{\mathcal{G}}_{n,d}\setminus \mathcal{G}_{n,d}'$ we have that $|A_M| = |B_M|$ for every maximum cut $M$ of $G$.

For each $G \in \hat{\mathcal{G}}_{n,d}\setminus \mathcal{G}_{n,d}'$, fix a maximum cut $M_G$ of $G$ which partitions $[n]$ into $A_G$ and $B_G$.
Then, for each $x\in A_G$ there exists $k \in [\lceil d/2 \rceil]$ such that $d_{M_G}(x) = \lfloor d/2 \rfloor + k$.
Let $\ell \in [\lceil d/2 \rceil]$ be such that there exist at least $(1-p)|\hat{\mathcal{G}}_{n,d}|/d$\COMMENT{By pigeonhole, there is a set of at least $2(1-p)|\hat{\mathcal{G}_{n,d}}|/(d+1)>(1-p)|\hat{\mathcal{G}}_{n,d}|/d$ graphs with this property.} graphs $G \in \hat{\mathcal{G}}_{n,d} \setminus \mathcal{G}_{n,d}'$ with the property that there are at least $n/(2d)$\COMMENT{By pigeonhole, in each maximum cut there are at least $\frac n2\frac{1}{\lceil d/2\rceil}\ge n/(2d)$ vertices in $A_G$ with the same degree in $M$.} vertices $x \in A_G$ with $d_{M_G}(x) = \lfloor d/2\rfloor+\ell$.
Let $D \coloneqq \lfloor d/2\rfloor+\ell$.
Denote the collection of all such graphs $G$ by $\Omega$.

For each $G\in\Omega$ and for each $x\in A_G$ such that $d_{M_G}(x)=D$, we consider all possible $x$-switchings of type $D$ where the $(x,D)$-switching configuration $\Lambda_{x,D}^+=(e_1,\ldots,e_D,f_1,\ldots,f_D)$ satisfies that $\{e_1,\ldots,e_D\}=E_{M_G}(x,B_G)$ and $\{f_1,\ldots,f_D\}\subseteq E_{G}(A_G)$.
We say that any $G'\in\mathcal{G}_{n,d}$ which can be obtained from $G$ by such an $x$-switching of type $D$, is obtained by an \emph{out-switching} from $G$, and we call $\Lambda_{x,D}^+$ an \emph{out-switching configuration}.
Let $\Omega'$ denote the set of all graphs $G'\in\mathcal{G}_{n,d}$ which can be obtained by out-switchings from some graph $G\in\Omega$.
In particular, note that for each $G'$ obtained from $G$ by an out-switching we may define $A'\coloneqq A_G\setminus\{x\}$ and $B'\coloneqq[n]\setminus A'$, so that $|A'|<|B'|$, $e_{G'}(u,B')>d/2$ for all $u\in A'$, and $e_{G'}(v,A')>d/2$ for all $v\in B'$, which means, as observed previously, that $G'$ is not $(d-1)/2$-resilient with respect to Hamiltonicity.
Therefore, $\Omega'\subseteq(\mathcal{G}_{n,d}\setminus\hat{\mathcal{G}}_{n,d})\cup\mathcal{G}_{n,d}'$ and $\Omega\cap\Omega'=\varnothing$.

To show that $\Omega'$ is large, we consider an auxiliary bipartite graph $\Gamma$ with parts $\Omega$ and $\Omega'$.
We place an edge between $G \in \Omega$ and $G' \in \Omega'$ if $G'$ is obtained from $G$ by an out-switching.
First, let $G \in \Omega$.
We will now provide a lower bound on the number of out-switchings from $G$.
Since $G \in \Omega$, by construction there are at least $n/(2d)$ vertices $x\in A_G$ such that $d_{M_G}(x)=D$.
For each such $x$, the number of out-switching configurations is given by the different choices for the edges in $(e_1,\ldots,e_D,f_1,\ldots,f_D)$, chosen sequentially.
There are $D!$ choices for $(e_1,\ldots,e_D)$.
For all $i\in[D]$, as each of the $f_i$ has to be independent from the previously chosen edges, at distance at least $2$ from $e_i$, and spanned by $A_G$, by \cref{prop: edgest} we conclude that the number of choices for $f_i$ is at least $n/100-4d^2$\COMMENT{By \cref{prop: edgest} there are at least $n/100$ edges inside of $A$. 
Each $f_i$ must be at distance at least $2$ from $e_i$. 
From $x$, this means that there are at most $d^2$ forbidden edges, while from the other endpoint of $e_i$ there are at most $d^2$ forbidden edges as well.
Furthermore, the $f_i$ must be independent.
For $f_1$ there are at least $n/100-2d^2$ choices. 
Each $f_i$ ``blocks'' at most $2(d-1)$ edges, so for each new $f_i$ that we add we ``lose'' at most $2d$ choices for the next one.
Therefore, for $f_i$ there are at least $n/100-2d^2-(i-1)2d\geq n/100-2d^2-2id$ choices.
As $i\leq D\leq d$, for all $f_i$ there are at least $n/100-4d^2$ choices.}.
Finally, once the out-switching configuration is given, there are $2^D$ possible switchings, one for each possible orientation of the set of edges $\{f_i:i\in[D]\}$; on the other hand, $D!$ different out-switching configurations result in the same outcome $G'$\COMMENT{Indeed, for any $\sigma\in\mathfrak{S}_\ell$ we have that $\Lambda_{x,D}^+=(e_1,\ldots,e_\ell,f_1,\ldots,f_\ell)$ and $\Lambda_{x,D}^{+,\sigma}=(e_{\sigma(1)},\ldots,e_{\sigma(\ell)},f_{\sigma(1)},\ldots,f_{\sigma(\ell)})$ give the same switchings.}.
We conclude that
\begin{equation}\label{eq: counterswitch2}
    d_\Gamma(G) \ge \frac{n}{2d}2^D\left(\frac{n}{100}-4d^2\right)^D.
\end{equation}

Now consider any $G' \in \Omega'$.
It is easy to see that\COMMENT{We will now provide an upper bound for the number of $x$-switchings of type $D$ from $G'$, which in turn provides an upper bound for the degree of $G'$ in $\Gamma$.
We first choose a vertex $x \in [n]$, for which there are at most $n$ choices.
We must now choose an $(x,D)$-switching configuration $\Lambda_{x,D}^+=(e_1,\ldots,e_D,f_1,\ldots,f_D)$ in $G'$.
For the $D$ edges $e_i$ incident to $x$ there are $d!/(d-D)!$ choices.
For the $D$ independent edges $f_i$ there are at most $(nd/2)^D$ choices.
Finally, for each $(x,D)$-switching configuration there are $2^D$ possible switchings, according to the orientations of the edges $f_i$ with $i\in[D]$, and for each switching configuration there are $D!$ different switching configurations which result in the same set of possible switchings.}
\begin{equation}\label{eq: counterswitch3}
    d_\Gamma(G') \leq n2^D\binom{d}{D}\left(\frac{nd}{2}\right)^D.
\end{equation}

Therefore, by double-counting the edges in $\Gamma$, from \eqref{eq: counterswitch2} and \eqref{eq: counterswitch3} we have that 
\[|\Omega| \leq 2d\binom{d}{D}\left(\frac{101d}{2}\right)^D|\Omega'|.\] 
It follows that there exists a constant $p$ which does not depend on $n$ for which a $p$ fraction of the graphs in $\hat{\mathcal{G}}_{n,d}$ are not $(d-1)/2$-resilient with respect to Hamiltonicity\COMMENT{Let $C\coloneqq2d^2\binom{d}{D}\left(\frac{101d}{2}\right)^D$.
We have that
\begin{align*}
    \frac{1-p}{d}|\hat{\mathcal{G}}_{n,d}|\leq|\Omega|\leq\frac{C}{d}|\Omega'|\leq\frac{C}{d}\left(o(1)|\hat{\mathcal{G}}_{n,d}|+|\mathcal{G}_{n,d}'|\right)=\frac{C}{d}(p+o(1))|\hat{\mathcal{G}}_{n,d}|\implies p\geq\frac{1-o(1)C}{1+C}\geq\frac{1}{2(C+1)}
\end{align*}
for $n$ sufficiently large.
}.
The result follows.
\end{proof}


\bibliographystyle{afstyle}
\bibliography{regular_resilience_bib}

\appendix

\section{Proof of \texorpdfstring{\cref{prop: small subgraph}}{Proposition~\ref{prop: small subgraph}}}\label{appendix}

In this section we make use of the following notation.
Let $\mathbf{d} = (d_1, \dots, d_n)$ be a graphic degree sequence and let $D\coloneqq\sum_{i=1}^{n} d_i$.
Let $X$ be the expanded set of $[n]$ with respect to the degree sequence $\mathbf{d}$.
For convenience, for a point $x \in X$ we refer to the unique $v \in [n]$ such that $x$ is an element of the expanded set of $v$ as the \emph{contracted vertex} of $x$.
For any $i\in[D/2]\cup\{0\}$, we call the edge set of a matching on $X$ with exactly $i$ edges an \emph{$i$-configuration}.
For each $i\in[D/2]\cup\{0\}$, let $\Omega_i$ be the collection of all $i$-configurations on $X$.
For each $i\in[D/2]\cup\{0\}$ and $M_i \in \Omega_i$, let $\Sigma(M_i)$ denote the collection of all orderings $(x_1, x_2, \dots, x_D)$ of $X$ for which $\{x_1, \dots, x_{2i-1}, x_{2i}\} = V(M_i)$.

For $i\in[D/2]$, we now describe two algorithms for obtaining $i$-configurations from $(i-1)$-configurations.
First, given any $M_{i-1}\in\Omega_{i-1}$, consider the set of all pairings $e\in X^{(2)}$ that can be added to $M_{i-1}$ so that $M_{i-1} \cup e\in\Omega_i$, and choose one such pairing, $e_i$, uniformly at random.
Then, output $M_i\coloneqq M_{i-1} \cup e_i$.
We refer to this process as \emph{Algorithm~A}.

To describe the second algorithm, for any $i\in[D/2]$, given $M_{i-1} \in \Omega_{i-1}$ and $\sigma \in \Sigma(M_{i-1})$, let
\[
  M_i\coloneqq\left\{
  \begin{array}{@{}ll@{}}
    M_{i-1} \cup \{x_{2i-1}x_{2i}\}, & \text{ with probability }1/(2i-1), \\
    (M_{i-1} \cup \{x_{2i-1}z_1, x_{2i}z_2\}) \setminus \{z_1z_2\} , & \text{ with probability }(2i-2)/(2i-1),
  \end{array}\right.
\]
where the pairing $z_1z_2$ is chosen uniformly at random from $M_{i-1}$ and $z_1$ is chosen uniformly at random from $\{z_1, z_2\}$.
We refer to the above rule for obtaining $M_i \in \Omega_i$ as \emph{Algorithm~B}.
Whenever Algorithm~B returns $M_i \coloneqq M_{i-1} \cup x_{2i-1}x_{2i}$, we say the algorithm made a \emph{Type~A} choice.
Otherwise, we say the algorithm made a \emph{Type~B} choice.
Given an $(i-1)$-configuration $M$ and an $i$-configuration $M'$, we write $M \stackrel{\sigma}{\sim} M'$ if $M'$ is obtainable as an output of a single iteration of Algorithm~B with inputs $M$ and $\sigma$.

\begin{proposition}\label{prop: algoB}
Let $\mathbf{d} = (d_1, \dots, d_n)$ be a graphic degree sequence, and let $D\coloneqq\sum_{i=1}^{n} d_i$.
Let $i\in[D/2]$, $M_{i-1} \in \Omega_{i-1}$ and $\sigma \in \Sigma(M_{i-1})$.
Then, given input $M_{i-1}$ and $\sigma$, Algorithm~B outputs a configuration $M_i \in \Omega_i$ uniformly at random from the set $\{M \in \Omega_i: M_{i-1} \stackrel{\sigma}{\sim}M\}$.
\end{proposition}

\COMMENT{
\begin{proof}
Let $\Omega\coloneqq \{M \in \Omega_i: M_{i-1} \stackrel{\sigma}{\sim}M\}$.
Consider the complete bipartite graph $\Gamma$ with parts $\{M_{i-1}\}$ and $\Omega$. 
Note that each $M \in \Omega$ is of the form $M = M_{i-1} + x_{2i-1}x_{2i}$, or $M = M_{i-1} - ab + ax_{2i-1} + bx_{2i}$, for some $ab \in E(M_{i-1})$.
Note that $d_\Gamma(M_{i-1}) = 2i-1$, since there are $2i-2$ choices for the ordered set $(a,b)$.
Given $M_{i-1}$ we note that Algorithm~B selects an edge incident to $M_{i-1}$ in $\Gamma$ uniformly at random, and hence $M_i$ is selected uniformly at random from the set $\Omega$.
\end{proof}}

It follows that, by initialising Algorithm~B with the empty configuration on $X$ and any (fixed) ordering of the points in $X$ as an input, we can obtain uniformly at random a configuration $M \in \Omega_D$, by calling on the algorithm $D$ times.
Furthermore, for $i\in[D/2]$ we can initialise Algorithm~A with the empty configuration on $X$ and run it for $i$ iterations, to obtain an element $M_i$ uniformly at random from the set $\Omega_{i}$. 
We can then call on Algorithm~B with input $M_i$ and arbitrary $\sigma \in \Sigma(M_i)$, and run it for $D/2-i$ iterations to obtain a configuration $M_{D/2}$ uniformly at random from the set of configurations $\Omega_{D/2}=\mathcal{C}^*_{n,\mathbf{d}}$.
This ability to switch from Algorithm~A to Algorithm~B is crucial for the proof of \cref{prop: small subgraph}. \COMMENT{The idea is the following.
If we just called on the simpler Algorithm~A to analyse our configuration process, we run into difficulty when the process is near the final few edges. 
It becomes hard to estimate the probability that the final edge added say, won't be an edge already contained in the subgraph $R$ which we condition on being present.
Algorithm~B is easier to analyse in this sense.
However, we cannot simply use Algorithm~B either for the following reason. 
In the initial stages of the process whenever Algorithm~B defaults to `Choice B' it is difficult to argue that the probability of selecting an edge already contained in $R$ is small. 
However after many iterations it becomes easier to estimate this probability as being very small.
Thus, we begin our process with Algorithm~A and then switch to Algorithm~B to help with the analysis.}

\begin{proof}[Proof of \cref{prop: small subgraph}]
Let $X$ be the expanded set of $[n]$ with respect to $\mathbf{d}$, and let $D\coloneqq nd - \sum_{i=1}^{n} d_i$.
We will generate a configuration $M \sim \mathcal{C}_{n,\mathbf{d}}^*$ in $D/2$ steps.
To do so, we apply Algorithm~A iteratively $D/3$ times, starting with the empty configuration.
We then switch to Algorithm~B for the remaining $D/6$ iterations.

Let $M_0$ be the empty configuration on $X$.
We iteratively call on Algorithm~A to obtain, for each $i\in[D/3]$, an $i$-configuration $M_i$.
Note that $R + \varphi(M_i)$ is not simple if Algorithm~A inserts some pairing which becomes a loop or a multiple edge in $R + \varphi(M_i)$.
We may assume that, to obtain an $i$-configuration given an $(i-1)$-configuration, Algorithm~A selects a point $u$ not yet covered by a pairing, and then selects a second uncovered point $v$ to pair with $u$.
In any iteration, once $u$ is chosen, there are at most $d^2$ choices of $v$ that would result in a loop or a parallel edge.\COMMENT{Let $v_u\in[n]$ be the contracted vertex of $u$.
For the loops, there are at most $d-1$ points with the same contracted vertex as $u$ which would yield loops. 
For the multiple edges, consider the points with the same contracted vertex as $u$ (there are at most $d-d_{v_u}$ such points).
In the worst case all of these have already been paired with points from expanded sets of different vertices.
It follows that $u$ cannot be in a pairing with any point whose contracted vertex is one of these, as otherwise this would result in a multiple edge.
This gives a total of at most $(d-d_{v_u})(d-1)$ potential bad points with respect to multiple edges in $\varphi(M_i)$.
Furthermore, $v$ should not lie in the expanded set of any vertex in $N_R(v_u)$, as this would create an edge parallel to one of the edges of $R$.
There are $d_{v_u}$ vertices in $N_R(v_u)$.
Therefore, there are at most $d_{v_u}(d-1)$ potential bad points with respect to multiple edges in $R$.
Adding these up, we have at most $d(d-1)$ bad points with respect to parallel edges.
Adding this to the number of points that might result in loops, there are at most $d^2$.}
Furthermore, note that after the final iteration of Algorithm A we still have $D/3$ unpaired points of $X$. 
That is, at every iteration of Algorithm~A there are at least $D/3$ choices for the point $v$.
It follows that 
\[\mathbb{P}[R + \varphi(M_i) \text{ is not simple} \mid R + \varphi(M_{i-1})\text{ is simple}] \le \frac{d^2}{D/3}\]
for all $i\in[D/3]$.
Therefore, \COMMENT{
We have $\mathbb{P}[R + \varphi(M_{D/3}) \text{ is simple}]> (1- \frac{d^2}{D/3})^{D/3}$. 
We can use that $1-x \ge e^{-1.1x}$ on this range of $x$ (for $n$ sufficiently large and since $\delta \le 1/10$) to get $\mathbb{P}[R + \varphi(M_{D/3}) \text{ is simple}]> e^{-3d^2/2}$.}
\begin{equation}\label{eqn2}
\mathbb{P}[R + \varphi(M_{D/3}) \text{ is simple}] > e^{-3d^2/2}.
\end{equation}

Next, let $\sigma \in \Sigma({M_{D/3}})$.
We now call on Algorithm~B iteratively $D/6$ times.
For each $D/3<i\leq D/2$, we use input $M_{i-1}$ and $\sigma$ to generate an $i$-configuration $M_i$.
Suppose Algorithm~B makes a choice of Type~B on the $i$-th iteration. 
This happens with probability $(2i-2)/(2i-1)$.
We claim that the probability that, in any given iteration, when the algorithm makes a choice of Type~B, an edge is added which will result in a loop or multiple edge in $R + \varphi(M_i)$ is at most $6d^2/D$.
To see this, note first that, to select an edge that will result in a loop, the algorithm must choose $z_1z_2\in E(M_{i-1})$ such that either $z_1$ has the same contracted vertex as $x_{2i-1}$ or $z_2$ has the same contracted vertex as $x_{2i}$.
Thus, there are at most $2(d-1)$ choices for $z_1z_2$ which can result in a loop.
We now count the pairings which the algorithm could add and which would result in a parallel edge in $R + \varphi(M_i)$.
First, consider the case where $z_1z_2$ is such that there is already a pairing $u'v' \in M_{i-1}$ between some point $u'\in X$ with the same contracted vertex as $x_{2i-1}$ and some $v'\in X$ with the same contracted vertex as $z_1$, or there is an edge in $E(R)$ between the contracted vertices of $x_{2i-1}$ and $z_1$.
This gives at most $(d-1)^2$ choices which would result in a parallel edge.
Similarly, by  considering $x_{2i}$ and $z_2$ there are at most another $(d-1)^2$ choices which would result in a parallel edge.
Thus, there are at most $2d^2$ choices of $z_1z_2$ which can result in a parallel edge or loop in $R + \varphi(M_i)$.
Note also that the choice of $z_1z_2$ is made uniformly at random from a set of size at least $D/3$.
The claim follows and we have that
\[\mathbb{P}[R + \varphi(M_i) \text{ is simple} \mid R + \varphi(M_{i-1})\text{ is simple}] \ge \frac{2i-2}{2i-1}\Big(1 - \frac{6d^2}{D}\Big)\]
for all $D/3<i\leq D/2$.
We have that $\prod_{i = D/3 + 1}^{D/2} ((2i-2)/(2i-1)) \ge 1/2$ and $(1- 6d^2/D)^{D/6} \ge 2e^{-3d^2/2}$.\COMMENT{The first inequality follows from a telescopic series. 
We have that 
\[\prod_{i = D/3 + 1}^{D/2} \frac{2i-2}{2i-1} \ge \prod_{i = D/3 + 1}^{D/2} \frac{2i-3}{2i-1} = \frac{2D/3-1}{2D/3+1}\frac{2D/3+1}{2D/3+3} \dots \frac{D-3}{D-1}=\frac{2D/3-1}{D-1} > 1/2.\]
The second inequality follows using the inequality that $1-x > e^{-1.1x}$ (since $\delta \le 1/10$ and so $x \sim d/n$).} 
Therefore, 
\begin{equation}\label{eqn1}
\mathbb{P}[R + \varphi(M_{D/2}) \text{ is simple} \mid R + \varphi(M_{D/3})\text{ is simple}] \ge e^{-3d^2/2}.
\end{equation}
The result follows by combining \eqref{eqn2} and \eqref{eqn1}.
\end{proof}

\end{document}